\documentclass[12pt,reqno]{amsart}

\usepackage{color}
\usepackage{amsmath, amsthm, amssymb}
\usepackage{amsfonts}
\usepackage{xcolor}
\usepackage[ansinew]{inputenc}
\usepackage[dvips]{epsfig}
\usepackage{graphicx}
\usepackage{epstopdf}
\usepackage[english]{babel}
\usepackage{enumerate}
\usepackage{hyperref}
\theoremstyle{plain}
\newtheorem{thm}{Theorem}[section]
\newtheorem{cor}[thm]{Corollary}
\newtheorem{lem}[thm]{Lemma}
\newtheorem{prop}[thm]{Proposition}

\theoremstyle{definition}
\newtheorem{defi}[thm]{Definition}

\theoremstyle{remark}
\newtheorem{rem}[thm]{Remark}

\numberwithin{equation}{section}

\newcommand{\de}{\partial}

\newcommand{\fls}{(-\Delta)^s}
\newcommand{\R}{\mathbb{R}}

\newcommand{\N}{\mathbb{N}}

\newcommand{\Sp}{\mathbb{S}}

\newcommand{\average}{{\mathchoice {\kern1ex\vcenter{\hrule height.4pt
width 6pt depth0pt} \kern-9.7pt} {\kern1ex\vcenter{\hrule
height.4pt width 4.3pt depth0pt} \kern-7pt} {} {} }}

\def\R{\mathbb{R}}

\begin{document}

\title[The fractional obstacle problem with critical drift]{The obstacle problem for \\ the fractional Laplacian with critical drift}

\author{Xavier Fernández-Real}

\address{ETH Z\"{u}rich, Department of Mathematics, Raemistrasse 101, 8092 Z\"{u}rich, Switzerland}

\email{xavierfe@math.ethz.ch}

\author{Xavier Ros-Oton}

\address{University of Texas at Austin, Department of Mathematics, 2515 Speedway, TX 78712 Austin, USA}

\email{ros.oton@math.utexas.edu}


\thanks{The first author is supported by ERC grant ``Regularity and Stability in Partial Differential Equations
(RSPDE)'' and a fellowship from ``Obra Social la Caixa''. The second author is supported by NSF Grant DMS-1565186 and by MINECO Grant MTM-2014-52402-C3-1-P (Spain).}

\begin{abstract}
We study the obstacle problem for the fractional Laplacian with drift, $\min\left\{\fls u + b \cdot \nabla u,\,u -\varphi\right\} = 0$ in $\R^n$, in the critical regime $s = \frac{1}{2}$.

Our main result establishes the $C^{1,\alpha}$ regularity of the free boundary around any regular point $x_0$, with an expansion of the form
\[
u(x)-\varphi(x) = c_0\big((x-x_0)\cdot e\big)_+^{1+\tilde\gamma(x_0)} + o\left(|x-x_0|^{1+\tilde\gamma(x_0)+\sigma}\right),
\]
\[
\tilde{\gamma}(x_0) = \frac{1}{2}+\frac{1}{\pi} \arctan (b\cdot e),
\]
where $e \in \Sp^{n-1}$ is the normal vector to the free boundary, $\sigma >0$, and $c_0> 0$.

We also establish an analogous result for more general nonlocal operators of order 1. In this case, the exponent $\tilde\gamma(x_0)$ also depends on the operator.
\end{abstract}
\maketitle

%\tableofcontents

%\addtocontents{toc}{\protect\setcounter{tocdepth}{1}}  %això és perquè no surtin les subsections al tableofcontents

\section{Introduction}

We consider the obstacle problem for the fractional Laplacian with drift,
\begin{equation}
\label{eq.pbintro}
\min\big\{ \fls u+ b\cdot \nabla u,\,u -\varphi \big\}  =  0\quad\textrm{in}\quad\R^n,
\end{equation}
where $b\in \R^n$, and $\varphi: \R^n \to \R$ is a smooth obstacle.

Problem \eqref{eq.pbintro} appears when considering optimal stopping problems for Lévy processes with jumps. In particular, this kind of obstacle problems are used to model prices of (perpetual) American options; see for example \cite{CF11, BFR15} and references therein for more details. See also \cite{Sal12} and \cite{KKP16} for further references and motivation on the fractional obstacle problem.

We study the regularity of solutions and the corresponding free boundaries for problem \eqref{eq.pbintro}. Note that the value of $s\in (0,1)$ plays an essential role. Indeed, if $s > \frac{1}{2}$, then the gradient term is of lower order with respect to $\fls$, and thus one expects solutions to behave as in the case $b \equiv 0$. When $s < \frac{1}{2}$ the leading term is $b\cdot\nabla u$ and thus one does not expect regularity results for \eqref{eq.pbintro}. Finally, in the borderline case $s = \frac{1}{2}$ there is an interplay between $b\cdot \nabla u$ and $(-\Delta)^{1/2}$, and one may still expect some regularity, but it becomes a delicate issue.

In this work we study this critical regime, $s = \frac{1}{2}$. 
As explained in detail below, we establish the $C^{1,\alpha}$ regularity of the free boundary near regular points, with a fine description of the solution at such points.

It is important to remark that, when $s=\frac12$, problem \eqref{eq.pbintro} is equivalent to the \emph{thin} obstacle problem in $\R^{n+1}_+$ with an \emph{oblique} derivative condition on $\{x_{n+1}=0\}$.
Thus, our results yield in particular the regularity of the free boundary for such problem, too.

%Also, the case $s=\frac12$ seems to be of particular interest in Mathematical Finance, since in this case the associated L\'evy process is the so-called Normal Inverse Gaussian (NIG) process; see 

\subsection{Known results}
The regularity of solutions and free boundaries for \eqref{eq.pbintro} was first studied in \cite{Sil07, CSS08} when $b = 0$. In \cite{CSS08}, Caffarelli, Salsa, and Silvestre established the optimal $C^{1,s}$ regularity for the solutions and $C^{1,\alpha}$ regularity of the free boundary around regular points. More precisely, they proved that given any free boundary point $x_0\in \de\{u = \varphi\}$, then
\begin{enumerate}[(i)]
\item  either
\[
0< cr^{1+s} \leq \sup_{B_r(x_0)}(u-\varphi ) \leq Cr^{1+s}
\]
\item or
\[
0\leq \sup_{B_r(x_0)} (u-\varphi) \leq Cr^2.
\]
\end{enumerate}
The set of points satisfying (i) is called the set of \emph{regular points}, and it was proved in \cite{CSS08} that this set is open and $C^{1,\alpha}$.

Later, the singular set --- those points at which the contact set has zero density --- was studied in \cite{GP09} in the case $s = \frac{1}{2}$. More recently, the regular set was proved to be $C^{\infty}$ in \cite{JN16, KRS16}; see also \cite{KPS15, DS16}. The complete structure of the free boundary was described in \cite{BFR15} under the assumption $\Delta \varphi \leq 0$. Finally, the results of \cite{CSS08} have been extended to a wide class of nonlocal elliptic operators in \cite{CRS16}.

All the previous results are for the case $b = 0 $. For the obstacle problem with drift \eqref{eq.pbintro}, Petrosyan and Pop proved in \cite{PP15} the optimal $C^{1,s}$ regularity of solutions in the case $s > \frac{1}{2}$. This result was obtained by means of an Almgren-type monotonicity formula, treating the drift as a lower order term. In \cite{GP16}, the same authors together with Garofalo and Smit Vega García establish  $C^{1,\alpha}$ regularity for the free boundary around regular points, again in the case $ s> \frac{1}{2}$. They do so by means of a Weiss-type monotonicity formula and an epiperimetric inequality. The assumption $s > \frac{1}{2}$ is essential in both works in order to treat the gradient as a lower order term.

\subsection{Main result} We study the obstacle problem with critical drift
\begin{equation}
\label{eq.obstpb_st}
\begin{array}{rcl}
\min\big\{ (-\Delta)^{1/2}u+ b\cdot \nabla u,\,u -\varphi \big\} & = & 0~~\textrm{ in }~~\R^n, \\ \lim_{|x|\to \infty} u (x)& =& 0.
\end{array}
\end{equation}
Here $b$ is a fixed vector in $\R^n$, and the obstacle $\varphi$ is assumed to satisfy
\begin{equation}
\label{eq.obst}
\varphi \textrm{ is bounded},~\varphi \in C^{2,1}(\R^n), \textrm{ and } \{\varphi > 0\}\Subset \R^n.
\end{equation}
The solution to \eqref{eq.obstpb_st} can be constructed as the smallest supersolution above the obstacle and vanishing at infinity.

Our main result reads as follows.
\begin{thm}
\label{thm.1}
Let $u$ be the solution to \eqref{eq.obstpb_st}, with $\varphi$ satisfying \eqref{eq.obst}, and $b\in \R^n$.

Let $x_0\in \de\{u = \varphi\}$ be any free boundary point. Then we have the following dichotomy:

\begin{enumerate}[\upshape(i)]
\item either
\[
0< cr^{1+\tilde\gamma(x_0)} \leq \sup_{B_r(x_0)} (u-\varphi) \leq Cr^{1+\tilde\gamma(x_0)},~~\quad \quad\tilde\gamma(x_0)\in (0,1),
\]
for all $r\in (0,1)$,
\item or
\[
~~~~~~~~~~~~~~~~~~~0\leq \sup_{B_r(x_0)} (u-\varphi) \leq C_\varepsilon r^{2-\varepsilon}\quad\quad\textrm{for all } \varepsilon > 0,~r \in (0,1).
\]
\end{enumerate}
Moreover, the subset of the free boundary satisfying ${\rm (i)}$ is relatively open and is locally $C^{1,\alpha}$ for some $\alpha > 0$.

Furthermore, $\tilde\gamma(x_0)$ is given by
\begin{equation}
\label{eq.tildegammafls}
\tilde\gamma(x_0) = \frac{1}{2}+\frac{1}{\pi} \arctan \big(b\cdot\nu(x_0)\big),
\end{equation}
where $\nu(x_0)$ denotes the unit normal vector to the free boundary at $x_0$ pointing towards $\{u > \varphi\}$. Finally, for every point $x_0$ satisfying ${\rm (i)}$ we have the expansion
\begin{equation}
\label{eq.expansion}
u(x)-\varphi(x) = c_0\Big((x-x_0)\cdot\nu(x_0)\Big)_+^{1+\tilde\gamma(x_0)} + o\left(\left|x-x_0\right|^{1+\tilde\gamma(x_0)+\sigma}\right)
\end{equation}
for some $\sigma > 0$, and $c_0 > 0$. The constants $\sigma$ and $\alpha$ depend only on $n$ and $\|b\|$.
\end{thm}

We think it is quite interesting that the growth around free boundary points (and thus, the regularity of the solution) depends on the orientation of the normal vector with respect to the free boundary. To our knowledge, this is the first example of an obstacle-type problem in which this happens.

The previous theorem implies that the solution is $C^{1,\gamma_b}$ at every free boundary point $x_0$, with
\begin{equation}
\label{eq.gammab}
\gamma_b := \frac{1}{2}- \frac{1}{\pi}\arctan(\|b\|).
\end{equation}
Nonetheless, the constants may depend on the point $x_0$ considered, so that if we want a uniform regularity estimate for $u$ we actually have the following corollary. It establishes almost optimal regularity of solutions.

\begin{cor}
\label{cor.1}
Let $u$ be the solution to \eqref{eq.obstpb_st} for a given obstacle $\varphi$ of the form \eqref{eq.obst}, and a given $b\in \R^n$. Let $\gamma_b$ given by \eqref{eq.gammab}. Then, for any $\varepsilon > 0$ we have
\[
\|u\|_{C^{1,\gamma_b - \varepsilon}(\R^n)}\leq C_\varepsilon,
\]
where $C_\varepsilon$ is a constant depending only on $n$, $\|b\|$, $\varepsilon$, and $\|\varphi\|_{C^{2,1}(\R^n)}$.
\end{cor}

In order to prove Theorem~\ref{thm.1} we proceed as follows. First, we classify convex global solutions to the obstacle problem by following the ideas in \cite{CRS16}. Then, we show the Lipschitz regularity of the free boundary at regular points, and using the results in \cite{RS16b} we find that the free boundary is actually $C^{1,\alpha}$. Finally, to prove \eqref{eq.expansion}-\eqref{eq.tildegammafls} we need to establish fine regularity estimates up to the boundary in $C^{1,\alpha}$ domains. This is done by constructing appropriate barriers and a blow-up argument in the spirit of \cite{RS16}. Notice that, since we do not have any monotonicity formula for problem \eqref{eq.obstpb_st}, our proofs are completely different from those in \cite{PP15, GP16}.

\subsection{More general nonlocal operators of order 1 with drift} We will show an analogous result for more general nonlocal operators of the form
\begin{equation}
\label{eq.L}
L u(x) = \int_{\R^n} \left(\frac{u(x+y)+u(x-y)}{2}-u(x) \right)\frac{\mu(y/|y|)}{|y|^{n+1}} dy,
\end{equation}
with
\begin{equation}
\label{eq.L.cond}
\mu \in L^\infty(\Sp^{n-1}) ~~ \textrm{ satisfying }~~ \mu (\theta)  =\mu (-\theta) ~~\textrm{and}~~ 0<\lambda \leq \mu \leq \Lambda.
\end{equation}
The constants $\lambda$ and $\Lambda$ are the ellipticity constants. Notice that the operators $L$ we are considering are of order $1$.

The obstacle problem in this case is, then,
\begin{equation}
\label{eq.obstpb}
\begin{array}{rcl}
\min\big\{-Lu+ b\cdot \nabla u, u -\varphi \big\}& = & 0~~\textrm{ in }~~\R^n, \\ \lim_{|x|\to \infty} u (x)& =& 0.
\end{array}
\end{equation}

Our main result reads as follows.

\begin{thm}
\label{thm.2}
Let $L$ be an operator of the form \eqref{eq.L}-\eqref{eq.L.cond}. Let $u$ be the solution to \eqref{eq.obstpb}, with $\varphi$ satisfying \eqref{eq.obst}, and $b\in \R^n$.

Let $x_0$ be any free boundary point, $x_0\in \de\{u = \varphi\}$. Then we have the following dichotomy:

\begin{enumerate}[\upshape(i)]
\item either
\[
0< cr^{1+\tilde\gamma(x_0)} \leq \sup_{B_r(x_0)} (u-\varphi) \leq Cr^{1+\tilde\gamma(x_0)},~~\quad \quad\tilde\gamma(x_0)\in (0,1),
\]
for all $r\in(0,1)$.
\item or
\[
~~~~~~~~~~~~~~~~~~~0\leq \sup_{B_r(x_0)} (u-\varphi) \leq C_\varepsilon r^{2-\varepsilon}\quad\quad\textrm{for all }\varepsilon > 0,~r\in(0,1).
\]
\end{enumerate}
Moreover, the subset of the free boundary satisfying ${\rm (i)}$ is relatively open and is locally $C^{1,\alpha}$ for some $\alpha > 0$.

Furthermore, the value of $\tilde\gamma(x_0)$ is given by
\begin{equation}
\label{eq.tildegamma}
\tilde\gamma(x_0) = \frac{1}{2}+\frac{1}{\pi} \arctan \left(\frac{b\cdot\nu(x_0)}{\chi(\nu(x_0))}\right),
\end{equation}
where $\nu(x_0)$ denotes the unit normal vector to the free boundary at $x_0$ pointing towards $\{u > \varphi\}$, and
\begin{equation}
\label{eq.chi}
\chi(e) = \frac{\pi}{2}\int_{\Sp^{n-1}} |\theta\cdot e| \mu(\theta)d\theta~~\quad\textrm{for}\quad e\in \Sp^{n-1}.
\end{equation}
Finally, for any point $x_0$ satisfying ${\rm (i)}$ we have the expansion
\[
u(x)-\varphi(x) = c_0\Big((x-x_0)\cdot\nu(x_0)\Big)_+^{1+\tilde\gamma(x_0)} + o\left(\left|x-x_0\right|^{1+\tilde\gamma(x_0)+\sigma}\right)
\]
for some $\sigma > 0$, and $c_0 > 0$. The constants $\sigma$ and $\alpha$ depend only on $n$, the ellipticity constants, and $\|b\|$.
\end{thm}

This result extends Theorem~\ref{thm.1}, and the dependence on the operator $L$ is reflected in \eqref{eq.tildegamma}. For the fractional Laplacian we have $\chi \equiv 1$, and thus \eqref{eq.tildegamma} becomes \eqref{eq.tildegammafls}.

We will also prove an analogous result to Corollary~\ref{cor.1} regarding the almost optimal regularity of solutions; see Corollary~\ref{cor.2}.

\subsection{Structure of the work}

We will focus on the proof of Theorem~\ref{thm.2}, from which in particular will follow Theorem~\ref{thm.1}. The paper is organised as follows.

In Section \ref{sec.2} we introduce the notation and give some preliminary results regarding nonlocal elliptic problems with drift. In Section~\ref{sec.3} we establish $C^{1,\tau}$ estimates for solutions to the obstacle problem with critical drift. In Section~\ref{sec.4} we classify convex global solutions to the problem. In Section~\ref{sec.5} we introduce the notion of regular points and we prove that blow-ups of solutions around such points converge to convex global solutions. In Section~\ref{sec.6} we prove $C^{1,\alpha}$ regularity of the free boundary around regular points. In Section~\ref{sec.7} we establish estimates up to the boundary for the Dirichlet problem with drift in $C^{1,\alpha}$ domains, in particular, finding an expansion of solutions around points of the boundary. In Section~\ref{sec.8} we combine the results from Sections~\ref{sec.6} and \ref{sec.7} to prove Theorems~\ref{thm.1} and \ref{thm.2}. Finally, in Section~\ref{sec.9}, we establish a non-degeneracy property at all points of the free boundary when the obstacle is concave near the coincidence set.

\section{Notation and preliminaries}
\label{sec.2}
We begin our work with a section of notation and preliminaries. Here, we recall some known results regarding nonlocal operators with drift, and we also find a 1-dimensional solution.

Throughout the work we will use the following function in order to avoid a heavy reading, $\gamma : \R \to (0,1)$, given by
\begin{equation}
\label{eq.gamma}
\gamma(t) := \frac{1}{2} + \frac{1}{\pi} \arctan \left(t\right).
\end{equation}

We next introduce some known results regarding the elliptic problem with drift that will be used. The first one is the following interior estimate.

\begin{prop}
\label{prop.intest}
Let $L$ be an operator of the form \eqref{eq.L}-\eqref{eq.L.cond}, and let $b\in \R^n$. Let $u$ solve
\[
(-L+b\cdot \nabla) u = f,\quad \textrm{in} \quad B_1,
\]
for some $f$. Then, if $f\in L^\infty(B_1)$, and for any $\varepsilon > 0$,
\[
[u]_{C^{1-\varepsilon}(B_{1/2})} \leq C \left(\|f\|_{L^\infty(B_1)} + \|u\|_{L^\infty(B_1)} + \int_{\R^n}\frac{|u(y)|}{1+|y|^{n+1}} dy\right),
\]
where $C$ depends only on $n$, $\varepsilon$, the ellipticity constants, and $\|b\|$.

\end{prop}
The proof of Proposition~\eqref{prop.intest} is given in \cite{Ser15} in case $b = 0$ (in the much more general context of fully nonlinear equations). The proof of \cite{Ser15} uses the main result in \cite{CL14}. The proof of Proposition~\ref{prop.intest} follows simply by replacing the use of the result \cite{CL14} in \cite{Ser15} by \cite[Theorem 7.2]{SS16} or \cite[Corollary 7.1]{CD16}.

We also need the following boundary Harnack inequality from \cite{RS16b}.

\begin{thm}[\cite{RS16b}]
\label{thm.bdharnack}
Let $U\subset\R^n$ be an open set, let $L$ be an operator of the form \eqref{eq.L}-\eqref{eq.L.cond}, and let $b\in \R^n$.

Let $u_1, u_2 \in C(B_1)$ be viscosity solutions to
\[
\left\{\begin{array}{rcll}
(-L+b\cdot\nabla) u_i & = &  0 & \quad \textrm{in}\quad U\cap B_1\\
u_i & =&  0 & \quad \textrm{in}\quad B_1\setminus U,\\
\end{array}\right.,\quad i = 1,2,
\]
and such that
\[
u_i \geq 0 \quad \textrm{in}\quad \R^n,\quad\quad \int_{\R^n} \frac{u_i(y)}{1+|y|^{n+1}} dy = 1,\quad i = 1,2.
\]

Then,
\[
0< c u_2 \leq u_1 \leq Cu_2\quad \textrm{in} \quad U\cap B_{1/2},
\]
for some constants $c$ and $C$ depending only on $n$, $\|b\|$, $U$, and the ellipticity constants.
\end{thm}

We will also need the following result.

\begin{thm}[\cite{RS16b}]
\label{thm.bdharnack2}
Let $U\subset\R^n$ be a Lipschitz set, let $L$ be an operator of the form \eqref{eq.L}-\eqref{eq.L.cond}, and let $b\in \R^n$.

Let $u_1, u_2 \in C(B_1)$ be viscosity solutions to
\[
\left\{\begin{array}{rcll}
(-L+b\cdot\nabla) u_i & = &  g_i & \quad \textrm{in}\quad U\cap B_1\\
u_i & =&  0 & \quad \textrm{in}\quad B_1\setminus U,\\
\end{array}\right.,\quad i = 1,2,
\]
for some functions $g_i \in L^\infty(U\cap B_1)$, $i = 1,2$. Assume also that
\[
u_i \geq 0 \quad \textrm{in}\quad \R^n,\quad\quad \int_{\R^n} \frac{u_i(y)}{1+|y|^{n+1}} dy = 1,\quad i = 1,2.
\]

Then, there exists $\delta > 0$ depending only on $n$, $U$, the ellipticity constants, and $\|b\|$ such that, if
\[
\|g_i\|_{L^\infty(U\cap B_1)} \leq \delta \quad \textrm{in}\quad U\cap B_1,\quad \quad\quad \quad\quad \quad i = 1,2,
\]
then
\[
\left\|\frac{u_1}{u_2}\right\|_{C^{\sigma}(U\cap B_{1/2})} \leq C,
\]
for some constants $\sigma$ and $C$ depending only on $n$, $U$, the ellipticity constants, and $\|b\|$.
\end{thm}

Finally, to conclude this section we study how 1-dimensional powers behave with respect to the operator, and in particular, we find a 1-dimensional solution to the problem. This solution is the same as the one that appears as a travelling wave solution in the parabolic fractional obstacle problem for $s = \frac{1}{2}$; see \cite[Remark 3.7]{CF11}.

\begin{prop}
\label{prop.1DL}
Let $b\in \R$, and let $u\in C(\R)$ be defined by
\[
u (x) := (x_+)^\beta,
\]
for $\beta \in (0,1)$. Then $u$ satisfies
\[
\begin{split}
(-\Delta)^{1/2} u + bu' = \beta\big(b\sin(\beta\pi)+\cos(\beta\pi)\big)(x_+)^{\beta-1}\quad \textrm{in}\quad \R_+,\\ u \equiv 0\quad \textrm{in}\quad \R_-.
\end{split}
\]

In particular, let us define
\[
u_0 (x) := C(x_+)^{\gamma(b)},
\]
where
\[
\gamma(t) := \frac{1}{2} + \frac{1}{\pi} \arctan \left(t\right) \in (0,1).
\]
Then, $u_0$ satisfies
\[
\begin{split}
(-\Delta)^{1/2} u_0 + bu_0' = 0\quad \textrm{in}\quad \R_+,\\ u_0 \equiv 0\quad \textrm{in}\quad \R_-,
\end{split}
\]
i.e., $u_0$ is a solution to the 1-dimensional non-local elliptic problem with critical drift and with zero Dirichlet conditions in $\R_-$.

\end{prop}
\begin{proof}
Define the harmonic extension to $\R^2_+$, $\bar u = \bar u(x,y)$, via the Poisson kernel, so that $\bar u(x, 0) = u(x)$, and $-\de_y \bar u(x, 0) = (-\Delta)^{1/2} u(x)$. We have that $\bar u$ solves,
\begin{equation}
\label{eq.expb}
\left\{\begin{array}{rcll}
\Delta \bar u& = &  0 & \quad \textrm{in}\quad \R^2\cap \{y > 0\}\\
\bar u & = &  0 & \quad \textrm{in} \quad \{x\leq 0\}\cap \{y = 0\}.\\
\end{array}\right.
\end{equation}

For simplicity, define the reflected function $w(x, y) = \bar u(-x,y)$, and let us consider that, by separation of variables in polar coordinates, $w(r, \theta) = g(r)h(\theta)$, for $r\geq 0$, $\theta\in [0,\pi]$ (we use the standard variables, $x = r\cos\theta$, $y = r\sin\theta$). Notice that we are considering homogeneous solutions, so that $g(r) = r^\beta$.  Then, from \eqref{eq.expb} we get
\begin{equation}
\label{eq.expb2}
\left\{\begin{array}{rcll}
g'' h +r^{-1}g'h + r^{-2}gh'' & = &  0 & \quad \textrm{in}\quad \{r> 0\}\cap \{\theta \in (0,\pi)\}\\
h(0) & = &  0 & \\
\end{array}\right.
\end{equation}
from which arise that $w$ can be expressed as
\[
w(r, \theta) = r^\beta \sin(\beta\theta).
\]

Now notice that, for $r > 0$,
\[
((-\Delta)^{1/2} u + bu')(r) = (r^{-1}\de_{\theta} +b\de_r)w(r, \theta)\bigr|_{\theta = \pi} = \beta\left(b\sin(\beta\pi)+\cos(\beta\pi)\right)r^{\beta-1}.
\]

Solving for $\beta$ we obtain that it is a solution for $\beta = \gamma(b)$. Moreover, notice that for $\beta <\gamma(b)$ it is a supersolution, and for $\beta >\gamma(b)$ a subsolution.
\end{proof}

\section{$C^{1,\tau}$ regularity of solutions}
\label{sec.3}
In this section we prove $C^{1, \tau}$ regularity of solutions to the obstacle problem with critical drift. For this, we use the method in \cite[Section 2]{CRS16}.

Throughout this section we can consider the wider class of nonlocal operators
\begin{equation}
\label{eq.L.2}
L u(x) = \int_{\R^n} \left(\frac{u(x+y)+u(x-y)}{2}-u(x) \right)\frac{a(y)}{|y|^{n+1}} dy,
\end{equation}
with
\begin{equation}
\label{eq.L.3}
a \in L^\infty (\R^n) ~~\textrm{ satisfying }~~ a (y) = a(-y) ~~\textrm{ and }~~ \lambda \leq a \leq \Lambda,
\end{equation}
so that we are dropping the homogeneity condition of the kernel.

\begin{lem}
\label{lem.basic}
Let $L$ be an operator of the form \eqref{eq.L.2}-\eqref{eq.L.3} and let $b\in \R^n$. Let $\varphi$ be any obstacle satisfying \eqref{eq.obst}, and let $u$ be a solution to \eqref{eq.obstpb}. Then,
\begin{enumerate}[(a)]

\item $u$ is semiconvex, with
\[
~~~~~~~~~~~~~~~~~~~~~~~~~~~~~~~~~~~\de_{ee} u  \geq  - \|\varphi\|_{C^{1,1}(\R^n)}  ~~\textrm{ for all }~~ e \in \Sp^{n-1}.
\]

\item $u$ is bounded, with
\[
\|u\|_{L^\infty(\R^n)}  \leq  \|\varphi\|_{L^\infty(\R^n)}.
\]

\item $u$ is Lipschitz, with
\[
\|u\|_{{\rm Lip}(\R^n)}  \leq  \|\varphi\|_{{\rm Lip}(\R^n)}.
\]
\end{enumerate}	
\end{lem}
\begin{proof}
The proof is exactly the same as in \cite[Lemma 2.1]{CRS16}, since the operator $-L+b\cdot \nabla$ still has maximum principle and is translation invariant.
\end{proof}

We next prove the lemma that will yield the $C^{1,\tau}$ regularity of solutions.

\begin{lem}
\label{lem.deltatau}
There exist constants $\tau > 0$ and $\delta > 0$ such that the following statement holds true.

Let $L$ be and operator of the form \eqref{eq.L.2}-\eqref{eq.L.3}, let $b\in \R^n$, and let $u\in {\rm Lip}(\R^n)$ be a solution to
\begin{equation*}
  \begin{array}{rcll}
  u &\geq &0& \textrm{in} ~~\R^n\\
  \de_{ee} u &\geq &-\delta & \textrm{in} ~~B_2~~ \textrm{ for all } e\in \Sp^{n-1}\\
  (-L+b\cdot\nabla)(u - u(\cdot - h))&\leq & \delta |h|& \textrm{in} ~~\{u > 0\}\cap B_2~~ \textrm{ for all } h\in \R^n,\\
  & & &\textrm{in the viscosity sense}. \\
  \end{array}
\end{equation*}
satisfying the growth condition
\[
\sup_{B_R} |\nabla u | \leq R^\tau ~\textrm{ for }~ R\geq 1.
\]
Assume that $u(0) = 0$. Then,
\[
|\nabla u ( x) | \leq 2 |x|^\tau.
\]

The constants $\tau$ and $\delta$ depend only on $n$, the ellipticity constants and $\|b\|$.
\end{lem}
\begin{proof}
The proof is very similar to that of \cite[Lemma 2.3]{CRS16}.

Define
\[
\theta(r) := \sup_{\bar r\ge r} \left\{ (\bar r)^{-\tau}\sup_{B_{\bar r}}  |\nabla u|\right\}
\]

Note that, by the growth control on the gradient, $\theta(r) \leq 1$ for $r \geq 1$. Note also that $\theta$ is nonincreasing by definition.

To get the desired result, it is enough to prove $\theta(r) \leq 2$ for all $r\in (0,1)$. Assume by contradiction that $\theta(r)> 2$ for some $r\in(0,1)$, so that from the definition of $\theta$, there will be some $\bar r\in(r,1)$ such that
\[
(\bar r)^{-\tau}\sup_{B_{\bar r}}  |\nabla u| \geq (1-\varepsilon) \theta(r) \geq (1-\varepsilon) \theta(\bar r) \geq \frac{3}{2},
\]
for some small $\varepsilon>0$ to be chosen later.

We now define
\[
\bar u(x) := \frac{u(\bar r x)}{\theta(\bar r) (\bar r)^{1+\tau}},
\]
and
\begin{align*}
L_{\bar r} w(x) := \int_{\R^n} \left(\frac{w(x+y)+w(x-y)}{2}-w(x) \right)\frac{a(\bar ry)}{|y|^{n+1}} dy
\end{align*}
Notice that $L_{\bar r}$ is still of the form \eqref{eq.L.2}-\eqref{eq.L.3}.

The rescaled function satisfies
\[\begin{array}{rcll}
\bar u&\geq&0\quad &\textrm{in}\ \R^n \\
D^2 \bar u&\geq& -(\bar r)^{2-1-\tau}\delta {\rm Id}  \geq -\delta {\rm Id} \quad &\textrm{in}\ B_{2/\bar r} \supset B_2 \\
(-L_{\bar r}+b\cdot \nabla) (\bar u- \bar u(\cdot-\bar h))&\leq&  (\bar r)^{-\tau}\delta |\bar r\bar h| \leq \delta |\bar h| \quad &\mbox{in } \{\bar u>0\}\cap B_{2} \\
& & & \mbox{for all }h\in\R^n,
\end{array}\]
Moreover, by definition of $\theta$ and $\bar r$, the rescaled function $\bar u$ also satisfies
\begin{equation}
\label{eq.growthc11}
1-\varepsilon \le  \sup_{|\bar h|\le 1/4} \sup_{B_1} \frac{\bar u- \bar u(\cdot-\bar h)}{|\bar h|}  \quad \mbox{and}\quad
 \sup_{|\bar h|\le 1/4} \sup_{B_R} \frac{\bar u-\bar u(\cdot-\bar h)}{|\bar h|} \le  (R+1/4)^\tau
\end{equation}
for all $R\ge 1$.

Let $\eta\in C^2_c(B_{3/2})$ with $\eta\equiv 1$ in $B_1$, $\eta \leq 1$ in $B_{3/2}$. Then,
\[
\sup_{|\bar h|\leq 1/4}  \sup_{{B_{3/2}}}  \left(\frac{\bar u- \bar u(\cdot-\bar h)}{|\bar h|} + 3\varepsilon\eta\right)\geq 1+2\varepsilon.
\]
Fix $h_0\in B_{1/4}$ such that
\[
t_0 := \max_{\overline{B_{3/2}}} \left(\frac{\bar u - \bar u(\cdot- h_0)}{| h_0|} +3\varepsilon\eta\right)\geq 1+\varepsilon.\]
and let  $x_0\in \overline{B_{3/2}}$ be such that
\begin{equation}
\label{eq.etatouches}
\frac{\bar u(x_0)- \bar u(x_0- h_0)}{| h_0|} + 3\varepsilon \eta(x_0) = t_0 .
\end{equation}

Let us denote
\[
v(x) : = \frac{\bar u(x) - \bar u(x - h_0)}{|h_0|}.
\]
Then, we have
\[
v + 3\varepsilon \eta \leq v(x_0) +3\varepsilon\eta(x_0) =  t_0  \quad \textrm{ in }\quad \overline{B_{3/2}}.
\]
Moreover, if $\tau$ is taken small enough then
\[\sup_{B_4} v \leq (4+1/4)^\tau<1+\varepsilon\leq t_0,\]
so that in particular $x_0$ is in the interior of $B_{3/2}$, and
\begin{equation}
\label{eq.etatouches2}
v +3\varepsilon\eta\le t_0\quad \mbox{in }\overline{B_3}.
\end{equation}
Note also that $x_0\in \{\bar u>0\}$ since otherwise $\bar u(x_0)-\bar u(x_0-h_0)$ would be a nonpositive number.

We now evaluate the equation for $v$ at $x_0$ to obtain a contradiction. To do so, recall that $D^2\bar u \geq -\delta {\rm Id}$ in $B_2$, $\bar u \geq 0$ in $\R^n$, and $\bar u(0)=0$.
It follows that, for $z\in B_2$ and $t'\in(0,1)$,
\[
\bar u(t'z) \leq t' \bar u(z)+ (1-t') \bar u(0)  +  \frac {\delta |z|^2} 2  t'(1-t') \leq \bar u(z) + \frac {\delta |z|^2} 2  t'(1-t')
\]
and thus, for $t\in(0,1)$, setting $z =x(1+ t/|x|)$ and $t' = 1/(1+t/|x|)$ we obtain, for $x\in B_1$,
\[
\bar u(x) -\bar u\left( x + t\frac{x}{|x|}\right)  \le \frac \delta 2 (|x|+t)^2 \frac{t/|x|}{(1+t/|x|)^2 } = \frac {\delta |x| t} 2 \le \delta t .
\]

Therefore, denoting $e = h_0 /|h_0|$, $t = |h_0|\le 1$ and using that by \eqref{eq.growthc11}, if $\tau$ small enough,
\[
\|\bar u\|_{\rm Lip(B_1)} \le \frac{4}{3},
\]
we obtain
\begin{equation}
\label{eq.conebd}
\begin{split}
v(x)= \frac{\bar u(x)- \bar u(x- te)}{t}
&\le  \frac{\bar u(x)- \bar u(x- te)}{t}  +  \frac{\bar u\left( x + t\frac{x}{|x|}\right) -\bar u(x) }{t}  + \delta
\\
&\le \frac{\bar u\left( x + t\frac{x}{|x|}\right)- \bar u(x- te)}{t}  +   \delta
\\
&\le   \frac{4}{3}\left| e +\frac{x}{|x|} \right|+\delta \,\le\, \frac{1} 4
\end{split}
\end{equation}
in $\mathcal C_e\cap B_1$ provided $\delta$ is taken smaller than $1/12$; where $\mathcal C_e$ is the cone,
\[
\mathcal C_e : = \left\{x\,:\, \left|e+ \frac{x}{|x|}\right| \le \frac 18 \right\}.
\]

On the other hand, we know that
\begin{equation}
\label{eq.boundvxy}
v(x_0 + y) - v(x_0) \leq 3\varepsilon\big(\eta(x_0) - \eta(x_0+y)\big) ~~\textrm{in}~~ B_3.
\end{equation}

This allows us to define
\[
\phi(x_0+y) =
\left\{ \begin{array}{rl}
  v(x_0) + 3\varepsilon\big(\eta(x_0) - \eta(x_0+y)\big) &\textrm{in }B_{1/8}\\
  v(x_0 + y)&\textrm{otherwise}.\\
  \end{array}\right.
\]

Notice that $\phi$ is regular around $x_0$ and that $\phi \geq v$ everywhere, and recall that $(- L_{\bar r}+b\cdot \nabla) v (x_0) \leq \delta$ in the viscosity sense. Therefore, we have
\begin{equation}
\label{eq.deltaineq}
- {L}_{\bar r} \phi(x_0)-C\|b\|\varepsilon \leq (- L_{\bar r}+b\cdot \nabla) \phi(x_0) \leq \delta.
\end{equation}

Now, using
\[
1-2\varepsilon\leq v(x_0)\leq 1+\varepsilon,
\]
and defining
\[
\delta \phi(x, y) := \frac{\phi(x+y)+\phi(x-y)}{2}-\phi(x),
\]
we can bound $\delta \phi(x_0, y)$ as
\[
\delta \phi(x_0, y) \leq \left\{\begin{array}{ll}
C\varepsilon |y|^2& \quad\textrm{in}\quad B_2\\[0.3cm]
(|y|+2)^\tau-1+2\varepsilon&\quad\textrm{in}\quad \R^n\setminus B_1\\[0.3cm]
-3/8+C\varepsilon&\quad\textrm{in}\quad (-x_0 + \mathcal C_e\cap B_1)\setminus B_{1/4}.
\end{array}\right.
\]

The first inequality follows because around $x_0$ and from \eqref{eq.boundvxy} we have the bound $\delta\phi(x_0,y)\leq \frac{3}{2}\varepsilon \left(2\eta(x_0) - \eta(x_0+y)-\eta(x_0-y)\right)$ and $\eta$ is a $C^2$ function. The second inequality follows from \eqref{eq.growthc11}, and using that $\frac{1}{2}\left(|x_0+y|+\frac{1}{4}\right)^\tau + \frac{1}{2}\left(|x_0-y|+\frac{1}{4}\right)^\tau \leq (|y|+2)^\tau$. For the third inequality, notice that
\begin{align*}
\delta \phi(x_0, y) & = \frac{v(x_0+y)-v(x_0)}{2} + \frac{v(x_0-y)-v(x_0)}{2}\\
& \leq \frac{1}{8} -\frac{1}{2} + \epsilon + C\varepsilon \leq -\frac{3}{8} + C\varepsilon\quad \textrm{ in } \quad \mathcal (-x_0 + C_e\cap B_1)\setminus B_{1/4},
\end{align*}
where we have used \eqref{eq.conebd} to bound the first term and \eqref{eq.boundvxy} to bound the second one. The constant $C$ depends only on the $\eta$, so it is independent of everything else.

We then find
\begin{align*}
{L}_{\bar r} \phi (x_0) \leq &~ \Lambda \int_{B_1} C\varepsilon |y|^2 |y|^{-n-1}dy + \Lambda\int_{\R^n\setminus B_1} \bigl\{(|y|+2)^\tau-1+2\varepsilon\bigr\} |y|^{-n-1}dy \\
& + \lambda \int_{(-x_0+\mathcal C_e\cap B_1)\setminus B_{1/4}} \left( -\frac{3}{8} + C\varepsilon \right)|y|^{-n-1}dy \\
\leq & ~ C\varepsilon + C\int_{\R^n\setminus B_{1/2}} \bigl\{(|y|+2)^\tau-1\bigr\} |y|^{-n-1}dy - c,
\end{align*}
with $c>0$ independent of $\delta$ and $\tau$ (for $\varepsilon$ small).

Thus, combining with \eqref{eq.deltaineq} we get
\begin{equation}
\label{eq.cdelta}
c -C\left( (\|b\| + 1) \varepsilon + \int_{\R^n\setminus B_{1/2}} \frac{(|y|+2)^\tau-1}{|y|^{n+1}}dy \right) \leq -C\|b\|\varepsilon - \tilde{L}_{\bar r} \phi(x_0) \leq \delta.
\end{equation}
If $\varepsilon$ and $\tau$ are taken small enough so that the left-hand side in \eqref{eq.cdelta} is greater than $c/2$, we get a contradiction for $\delta \leq c/4$.
\end{proof}

The following proposition implies that the solution to the obstacle problem \eqref{eq.obstpb} is $C^{1,\tau}$ for some $\tau > 0$.

\begin{prop}
\label{prop.reg.u}
Let $L$ be any operator of the form \eqref{eq.L.2}-\eqref{eq.L.3}, let $b\in \R^n$, and let $u\in{\rm Lip}(\R^n)$ with $u(0) = 0$ be any function satisfying, for all $h\in\R^n$ and $e\in \Sp^{n-1}$, and for some $\varepsilon > 0$,
\[
\begin{array}{rcll}
u&\geq&0\quad &\textrm{in}\ \R^n \\
\partial_{ee} u&\geq& -K\quad &\textrm{in}\ B_2 \\
(-L+b\cdot\nabla) (u-u(\cdot-h))& \leq & K|h|\quad &\textrm{in}\ \{u>0\}\cap B_2\\
|\nabla u|&\leq& K(1+|x|^{1-\varepsilon}) \quad &\textrm{in}\ \R^n.
\end{array}
\]
Then, there exists a small constant $\tau>0$ such that
\[
\|u\|_{C^{1,\tau}(B_{1/2})}\leq CK.
\]
The constants $\tau$ and $C$ depend only on $n$, $\|b\|$, $\varepsilon$, and the ellipticity constants.
\end{prop}
\begin{proof}
The proof is standard and it is exactly the same as the proof of \cite[Proposition 2.4]{CRS16} by means of Lemma~\ref{lem.deltatau}.
\end{proof}

\section{Classification of convex global solutions}
\label{sec.4}

In this section we prove the following theorem, that classifies all convex global solutions to the obstacle problem with critical drift.

\begin{thm}
\label{thm.clas}
Let $L$ be an operator of the form \eqref{eq.L}-\eqref{eq.L.cond}. Let $\Omega\subset\R^n$ be a closed convex set, with $0\in \Omega$. Let $u\in C^1(\R^n)$ a function satisfying, for all $h\in \R^n$,
\begin{equation}
\label{eq.clas}
\left\{\begin{array}{rcll}
(-L+b\cdot\nabla) (\nabla u)  & = &  0 & \quad \textrm{in}\quad \R^n\setminus \Omega\\
(-L+b\cdot\nabla) (u - u(\cdot - h)) & \leq &  0 & \quad \textrm{in}\quad \R^n \setminus \Omega\\
D^2 u &  \geq & 0 & \quad \textrm{in} \quad \R^n\\
u & = & 0 & \quad \textrm{in} \quad \Omega\\
u & \geq & 0 & \quad \textrm{in} \quad \R^n.\\
\end{array}\right.
\end{equation}
Assume also the following growth control satisfied by $u$,
\begin{equation}
\label{eq.clas2}
\|\nabla u \|_{L^\infty(B_R)} \leq R^{1-\varepsilon}\quad \textrm{ for all }\quad R \geq 1,
\end{equation}
for some $\varepsilon > 0$. Then, either $u \equiv 0$, or
\begin{equation}
\label{eq.clas3}
\Omega = \{e\cdot x \leq 0 \} \quad \textrm{and}\quad u(x) = C(e\cdot x)_+^{1+\gamma(b\cdot e/\chi(e))},
\end{equation}
for some $e\in \Sp^{n-1}$ and $C > 0$. The value of $\chi(e)$ is given by \eqref{eq.chi} with the kernel $\mu$ of $L$, and $\gamma$ is given by \eqref{eq.gamma}.
\end{thm}

We start by proving the following proposition.

\begin{prop}
\label{prop.bdharnack}
Let $\Sigma$ be a non-empty closed convex cone, and let $L$ be an operator of the form \eqref{eq.L}-\eqref{eq.L.cond}. Let $u_1$ and $u_2$ be two non-negative continuous functions satisfying
\[
\int_{\R^n} \frac{u_i(y)}{1+|y|^{n+1}} dy < \infty,\quad i = 1,2.
\]

Assume, also, that they are viscosity solutions to
\[
\left\{\begin{array}{rcll}
(-L+b\cdot\nabla) u_i & = &  0 & \quad \textrm{in}\quad \R^n\setminus \Sigma\\
u_i & =&  0 & \quad \textrm{in}\quad \Sigma\\
u_i & > & 0 & \quad \textrm{in} \quad \R^n\setminus \Sigma.
\end{array}\right.
\]
Then,
\[
u_1 \equiv Ku_2 \quad \textrm{in}\quad \R^n,
\]
for some constant $K$.
\end{prop}
\begin{proof}
The proof is the same as the proof of \cite[Theorem 3.1]{CRS16}, using the boundary Harnack inequality in Theorem~\ref{thm.bdharnack}.

Suppose, without loss of generality, that $\Sigma \subsetneq \R^n$. Take $P$ a point with $|P| = 1$ and $B_r(P)\subset\R^n\setminus\Sigma$ for some $r > 0$, and assume that $u_i (P) = 1$. We want to prove $u_1 \equiv u_2$.

Define, given $R \geq 1$,
\[
\bar u_i(x) = \frac{u_i(Rx)}{C_i},
\]
with $C_i$ such that $\int_{\R^n} \bar u_i(y)(1+|y|)^{-n-1} dy = 1$. Thus, by Theorem~\ref{thm.bdharnack} there exists some $c > 0$ such that
\begin{equation}
\label{eq.comparable}
\bar u_1 \geq c \bar u_2 \quad\textrm{and}\quad \bar u_2 \geq c \bar u_1\quad\textrm{in}\quad B_{1/2}.
\end{equation}

In particular, $\bar u_1 (P/R)$ and $\bar u_2 (P/R)$ are comparable, so that $C_1$ and $C_2$ are comparable. Thus, from \eqref{eq.comparable},
\[
u_1 \geq c u_2 \quad\textrm{and}\quad u_2 \geq c u_1\quad\textrm{in}\quad B_{R/2},
\]
for any $R\geq 1$, so that the previous inequalities are true in $\R^n$.

Now take
\[
\bar c := \sup\{c > 0 : u_1 \geq cu_2 \quad\textrm{in}\quad \R^n\} < \infty.
\]

Define
\[
v = u_1 - \bar c u_2 \geq 0.
\]
Either $v \equiv 0$ in $\R^n$ or $v  >0$ in $\R^n\setminus\Sigma$ by the strong maximum principle. If $v \equiv 0$ we are done, because in this case $\bar c = 1$ due to the fact that $u_1(P) = u_2(P) = 1$.

Let us assume then that $v  >0$ in $\R^n\setminus\Sigma$. Apply the first part of the proof to $v/v(P)$ and $u_2$ to deduce that, for some $\delta > 0$, $v > \delta u_2$. This contradicts the definition of $\bar c$, so $v \equiv 0$ as we wanted.
\end{proof}

We can now prove the classification of convex global solutions in Theorem~\ref{thm.clas}

\begin{proof}[Proof of Theorem~\ref{thm.clas}]
First, by the same blow-down argument in \cite[Theorem 4.1]{CRS16}, we can restrict ourselves to the case in which $\Omega = \Sigma$ for $\Sigma$ a closed convex cone in $\R^n$ with vertex at 0.

We now split the proof into two cases:

{\it Case 1:} When $\Sigma$ has non empty interior there are $n$ linearly independent unitary vectors $e_i$ such that $-e_i\in \Sigma$. Define
\[
v_i := \de_{e_i} u,
\]
and note that, since $D^2 u \geq 0$ and $-e_i\in \Sigma = \{u = 0\}$, we have
\begin{equation}
\label{eq.clas4}
\left\{\begin{array}{rcll}
(-L+b\cdot\nabla) v_i  & = &  0 & \quad \textrm{in}\quad \R^n\setminus \Sigma\\
v_i& = &  0 & \quad \textrm{in}\quad \Sigma\\
v_i&  \geq & 0 & \quad \textrm{in} \quad \R^n.\\
\end{array}\right.
\end{equation}

From Proposition~\ref{prop.bdharnack}, we must have $v_i = a_i v_k$ for some $1 \leq k \leq n$, $a_i\in \R$, and for all $i = 1,\dots,n$, so that $\de_{e_i - a_ie_k}u \equiv 0$ in $\R^n$ for all $i \neq k$. Thus, there exists a non-negative function $\phi: \R\to \R$, $\phi \in C^1$, such that $u = \phi(e\cdot x)$ for some $e\in \Sp^{n-1}$; so that, since $0\in \de\Sigma$, $\Sigma = \{e\cdot x \leq 0\}$.

Notice that $\phi'\geq 0$ solves $(-L+(b\cdot e)\de)(\phi') = 0$ in $\R_+$ and $\phi' \equiv 0$ in $\R_-$, with the growth $\phi' (t)\leq C(1+t^{{1-\varepsilon}})$. From \cite[Lemma 2.1]{RS14}, we have
\[
(\chi(e) (-\Delta)^{1/2}+(b\cdot e)\de)(\phi') = 0 \quad\textrm{in}\quad \R_+,
\]
where $\chi(e)$ is given by \eqref{eq.chi}. Now, a non-negative solution to the previous equation is given by Proposition~\ref{prop.1DL}. Such solution is unique up to a multiplicative constant thanks to Proposition~\ref{prop.bdharnack}. Indeed, notice that the hypotheses of the lemma are fulfilled due to the growth control of $\phi'$ and the fact that $\phi'\geq 0$. Thus, we obtain
\[
\phi(t) = (t_+)^{1+\gamma(b\cdot e)/\chi(e)}\quad\textrm{for}\quad t \in \R,
\]
where $\gamma$ and $\chi$ are given by \eqref{eq.gamma} and \eqref{eq.chi} respectively.

{\it Case 2:} If $\Sigma$ has empty interior then by convexity it must be contained in some hyperplane $H= \{x\cdot e = 0\}$. From Proposition~\ref{prop.reg.u}, rescaling,
\[
[\nabla u]_{C^\tau(B_R)}\leq C(R),
\]
for some constant $C(R)$ depending on $R$; and for any $R \geq 1$. In particular, for any $h\in \R^n$, if we define
\[
v(x) = u(x)- u(x-h)\quad\textrm{for}\quad x\in \R^n,
\]
then $v\in C^{1,\tau}_{{\rm loc}}(\R^n)$. This implies that $(-L+b\cdot\nabla)v\in C^{\tau}_{{\rm loc}}(\R^n)$, but we already knew that $(-L+b\cdot\nabla)v = 0$ in $\R^n\setminus H$, so we must have
\[
(-L+b\cdot\nabla)v = 0\quad\textrm{in}\quad\R^n.
\]
Now, from the interior estimates in Proposition~\ref{prop.intest} rescaled on balls $B_R$ we have
\[
R^{1-\varepsilon/2}[v]_{C^{1-\varepsilon/2}(B_{R/2})} \leq C\left(\|v\|_{L^\infty(B_R)} + \int_{\R^n}\frac{|v(Ry)|}{1+|y|^{n+1}}dy\right).
\]
On the other hand, from the growth control on the gradient, we have
\[
\|v\|_{L^\infty(B_R)} \leq |h|R^{1-\varepsilon}.
\]
Putting the last two expressions together we reach
\[
[v]_{C^{1-\varepsilon/2}(B_{R/2})} \leq \frac{C|h|}{R^{\varepsilon/2}}.
\]

Now let $R\to \infty$ to obtain that $v$ must be constant for all $h$. That means that $u$ is affine, but $u(0) = 0$ and $u\geq 0$ in $\R^n$, so $u\equiv 0$.
\end{proof}

\section{Blow-ups at regular points}
\label{sec.5}
By subtracting the obstacle if necessary and dividing by $C\|\varphi\|_{C^{2,1}(\R^n)}$, we can assume that we are dealing with the following problem,
\begin{equation}
\label{eq.pb}
\left\{\begin{array}{rcll}
u & \geq &  0 & \quad \textrm{in}\quad \R^n\\
(-L+b\cdot\nabla) u & \leq & f & \quad \textrm{in}\quad \R^n \\
(-L+b\cdot\nabla) u & = & f & \quad \textrm{in}\quad \{u>0\} \\
D^2u & \geq & -{\rm Id} & \quad \textrm{in} \quad \R^n.\\
\end{array}\right.
\end{equation}
Moreover, dividing by a bigger constant if necessary, we can also assume that
\begin{equation}
\label{eq.pb2}
\|f\|_{C^1(\R^n)} \leq 1,
\end{equation}
and that
\begin{equation}
\label{eq.pb3}
\|u\|_{C^{1,\tau}(\R^n)}\leq 1.
\end{equation}
The validity of the last expression and the constant $\tau$ come from Proposition~\ref{prop.reg.u} and Lemma~\ref{lem.basic}.

Let us now introduce the notion of \emph{regular} free boundary point.

\begin{defi}
We say that $x_0 \in \de\{u > 0\}$ is a \emph{regular} free boundary point with exponent $\varepsilon$ if
\[
\limsup_{r\downarrow 0} \frac{\| u\|_{L^\infty(B_r(x_0))}}{r^{2-\varepsilon}} = \infty
\]
for some $\varepsilon > 0$.
\end{defi}

The following proposition states that an appropriate blow up sequence of the solution around a regular free boundary point converges in $C^1$ norm to a convex global solution.

\begin{prop}
\label{prop.regpt}
Let $L$ be an operator of the form \eqref{eq.L}-\eqref{eq.L.cond}, and let $b\in \R^n$. Let $u$ be a solution to \eqref{eq.pb}-\eqref{eq.pb2}-\eqref{eq.pb3}. Assume that $0$ is a regular free boundary point with exponent $\varepsilon$.

Then, given $\delta > 0$, $R_0 \geq 1$, there exists $r > 0$ such that the rescaled function
\[
v(x):= \frac{u(rx)}{r\|\nabla u\|_{L^\infty(B_r)}}
\]
satisfies
\[
\|\nabla v \|_{L^\infty(B_R)} \leq 2R^{1-\varepsilon}\quad \textrm{for all}\quad R\geq1,
\]
\[
\big|(-L+b\cdot\nabla)(\nabla v)\big|\leq \delta\quad \textrm{in}\quad \{v>0\},
\]
and
\[
|v-u_0| + |\nabla v - \nabla u_0|\leq \delta\quad\textrm{in}\quad B_{R_0},
\]
for some $u_0$ of the form \eqref{eq.clas3} and with $\|\nabla u_0\|_{L^\infty(B_1)} = 1$.
\end{prop}

Before proving the previous proposition, let us prove the following lemma.
\begin{lem}
\label{lem.regpt}
Assume $u \in C^1(B_1)$ satisfies $\|\nabla u\|_{L^\infty(\R^n)} = 1$, $u(0) = 0$, and
\[
\sup_{\rho \leq r} \frac{\|u\|_{L^\infty(B_r)}}{r^{2-\varepsilon}} \to \infty\quad \textrm{as} \quad \rho \downarrow 0.
\]
Then, there exists a sequence $r_k\downarrow 0$ such that $\|\nabla u\|_{L^\infty(B_{r_k})} \geq \frac{1}{2}r_k^{1-\varepsilon}$, and for which the rescaled functions
\[
u_k(x) = \frac{u(r_k x)}{r_k \|\nabla u \|_{L^\infty(B_{r_k})}}
\]
satisfy
\[
|\nabla u_k(x)| \leq 2(1+|x|^{1-\varepsilon})\quad\textrm{in}\quad \R^n.
\]
\end{lem}
\begin{proof}
Define
\[
\theta(\rho) := \sup_{r \geq \rho} \frac{\|\nabla u\|_{L^\infty(B_{r})}
}{r^{1-\varepsilon}}.
\]
Notice that, since $u(0) = 0$, we have
\[
\frac{\|u\|_{L^\infty(B_{r})}}{r^{2- \varepsilon}}\leq \frac{\|\nabla u\|_{L^\infty(B_{r})}}{r^{1-\varepsilon}}.
\]
Therefore, $\theta(\rho) \to \infty$ as $\rho \downarrow 0$, and notice also that $\theta$ is non-increasing.

Now, for every $k\in \N$, there is some $r_k\geq \frac{1}{k}$ such that
\begin{equation}
\label{eq.theta}
r_k^{\varepsilon-1}\|\nabla u \|_{L^\infty(B_{r_k})} \geq \frac{1}{2}\theta\left(1/k\right) \geq \frac{1}{2} \theta(r_k).
\end{equation}

Since $\|\nabla u\|_{L^\infty(\R^n)} = 1$, then
\[
r_k^{\varepsilon-1} \geq \frac{1}{2}\theta(1/k)\to \infty \quad\textrm{as}\quad k \to \infty,
\]
so that $r_k \to 0$ as $k\to \infty$. We also have $\theta(r_k)\geq 1$, and therefore $\|\nabla u\|_{L^\infty(B_{r_k})} \geq \frac{1}{2}r_k^{1-\varepsilon}$.

Finally, from the definition of $\theta$ and \eqref{eq.theta}, and for any $R \geq 1$, we have
\[
\|\nabla u_k\|_{L^\infty(B_R)} = \frac{\|\nabla u\|_{L^\infty(B_{r_kR})}}{\|\nabla u\|_{L^\infty(B_{r_k})}} \leq \frac{\theta(r_kR)(r_kR)^{1-\varepsilon}}{\frac{1}{2}(r_k)^{1-\varepsilon}\theta(r_k)} \leq 2R^{1-\varepsilon},
\]
which follows from the monotonicity of $\theta$.
\end{proof}

We can now prove Proposition~\ref{prop.regpt}, which follows taking the  sequence of rescalings given by Lemma~\ref{lem.regpt} together with a compactness argument.

\begin{proof}[Proof of Proposition~\ref{prop.regpt}]
Let $r_k\downarrow 0$ be the sequence given by Lemma~\ref{lem.regpt}. Therefore, the functions
\[
v_k(x) = \frac{u(r_k x)}{r_k\|\nabla u\|_{L^\infty(B_{r_k})}}
\]
satisfy
\[
\|\nabla v_k\|_{L^\infty(B_R)} \leq 2R^{1-\varepsilon}\quad\textrm{for all}\quad R\geq 1,
\]
and
\[
\|\nabla v_k\|_{L^\infty(B_1)} = 1,\quad v_k(0) = 0.
\]

Moreover,
\[
D^2 v_k = \frac{r_k}{\|\nabla u\|_{L^\infty(B_{r_k})}} D^2 u \geq -\frac{r_k}{\|\nabla u\|_{L^\infty(B_{r_k})}} {\rm Id},
\]
and, in $\{v_k > 0\}$,
\begin{align*}
\big|(-L+b\cdot\nabla)(\nabla v_k)\big| & = \frac{r_k}{\|\nabla u\|_{L^\infty(B_{r_k})}} \big|(-L+b\cdot\nabla)(\nabla u)\big| \\
& \leq  \frac{r_k}{\|\nabla u\|_{L^\infty(B_{r_k})}} \|\nabla f\|_{L^\infty} \leq  \frac{r_k}{\|\nabla u\|_{L^\infty(B_{r_k})}}.
\end{align*}

Notice that, from \eqref{eq.theta} and with the notation from the proof of Lemma~\ref{lem.regpt},
\[
\frac{1}{\eta_k} := \frac{\|\nabla u\|_{L^\infty(B_{r_k})}}{r_k} \geq \frac{\theta(r_k)}{2r^{\varepsilon}_k} \to \infty,\quad\textrm{as}\quad r_k\downarrow 0.
\]
Thus, in all we have a sequence $v_k$ such that $v_k\in C^1$, $v_k(0) = 0$, and
\[
\|\nabla v_k \|_{L^\infty(B_R)} \leq 2R^{1-\varepsilon}\quad \textrm{for all}\quad R\geq1,
\]
\[
\big|(-L+b\cdot\nabla)(\nabla v_k)\big|\leq \eta_k\quad \textrm{in}\quad \{v_k>0\},
\]
\[
D^2 v_k \geq -\eta_k {\rm Id},
\]
with $\eta_k\downarrow 0$. From the estimates in Proposition~\ref{prop.reg.u},
\[
\|\nabla v_k\|_{C^{\tau}(B_R)} \leq C(R)\quad\textrm{for all}\quad R \geq 1,
\]
for some constant depending on $R$, $C(R)$. Thus, up to taking a subsequence, $v_k$ converges in $C^1_{\rm loc}(\R^n)$ to some $v_\infty$ which by stability of viscosity solutions is a convex global solution to the obstacle problem \eqref{eq.clas} fulfilling \eqref{eq.clas2}.

By the classification theorem, Theorem~\ref{thm.clas}, $v_\infty$ must be of the form \eqref{eq.clas3}. Taking limits
\[
\|\nabla v_\infty\|_{L^\infty(B_1)} = 1
\]
and $v_\infty(0) = 0$. Now the result follows because $\eta_k\downarrow 0$ and $v_k$ converge in $C^1_{\rm loc}(\R^n)$ to $v_\infty$.
\end{proof}

\section{$C^{1,\alpha}$ regularity of the free boundary around regular points}
\label{sec.6}
In this section we prove $C^{1,\alpha}$ regularity of the free boundary around regular points.

We begin by proving the Lipschitz regularity of the free boundary, as stated in the following proposition.

\begin{prop}
\label{prop.fblip}
Let $L$ be an operator of the form \eqref{eq.L}-\eqref{eq.L.cond}, and let $b\in \R^n$. Let $u$ be a solution to \eqref{eq.pb}-\eqref{eq.pb2}-\eqref{eq.pb3}. Assume that $0$ is a regular free boundary point.

Then, there exists a vector $e\in \Sp^{n-1}$ such that for any $\ell > 0$, there exists an $r > 0$ and a Lipschitz function $g:\R^{n-1}\to \R$ such that
\[
\{u > 0\}\cap B_r = \big\{y_n> g(y_1,\dots,y_{n-1})\big\}\cap B_r,
\]
where $y = Rx$ is a change of coordinates given by a rotation $R$ with $Re = e_n$, and $g$ fulfils
\[
\|g\|_{{\rm Lip} (B_r)} \leq \ell.
\]
Moreover, $\de_{e'} u \geq 0$ in $B_{r}$ for all $e'\cdot e \geq \frac{\ell}{\sqrt{1+\ell^2}}$.
\end{prop}

The following lemma will be needed in the proof, and it is analogous to \cite[Lemma 6.2]{CRS16}.

\begin{lem}
\label{lem.lipreg}
There exists $\eta = \eta(n,\Lambda,\lambda, \|b\|)$ such that the following statement holds.

Let $L$ be an operator of the form \eqref{eq.L}-\eqref{eq.L.cond}, and let $b\in\R^n$. Let $E\subset B_1$ be relatively closed, and assume that, in the viscosity sense, $w\in C(B_1)$ satisfies
\begin{equation}
\left\{\begin{array}{rcll}
(-L+b\cdot \nabla)w & \geq & -\eta & \quad \textrm{in}\quad B_1\setminus E\\
w & = & 0 & \quad \textrm{in}\quad E\cup(\R^n\setminus B_2) \\
w & \geq & -\eta & \quad \textrm{in}\quad B_2 \setminus E,\\
\end{array}\right.
\end{equation}
and
\[
\int_{B_1} w_+ \geq 1.
\]

Then, $w$ is non-negative in $B_{1/2}$, i.e.,
\[
w \geq 0\quad\textrm{in}\quad B_{1/2}.
\]
\end{lem}
\begin{proof}
Let us argue by contradiction, and suppose that the statement does not hold for any $\eta > 0$. Define $\psi\in C_c^2(B_{3/4})$ be a radial function with $\psi \geq 0$, $\psi \equiv 1$ in $B_{1/2}$ and with $|\nabla \psi|\leq C(n)$. Let
\[
\psi_t(x) := - \eta -t+\eta\psi(x).
\]

If $w$ attains negative values on $B_{1/2}$, then there exists some $t_0 > 0$ and $z\in B_{3/4}$ such that $\psi_{t_0}$ touches $w$ from below at $z$, i.e. $\psi_{t_0} \leq w$ everywhere and $\psi_{t_0}(z) = w(z) < 0$. Let $\delta > 0$ be such that $w < 0$ in $B_\delta(z)$ (recall $w$ continuous). Let us now define
\begin{equation}
\bar w(x) := \left\{\begin{array}{ll}
w(x) & \quad \textrm{if}\quad x\in \R^n\setminus B_\delta(z)\\
\psi_{t_0}(x) & \quad \textrm{if}\quad x\in B_\delta(z).\\
\end{array}\right.
\end{equation}

Notice that $\bar w$ is $C^2$ around $z$, and is such that $\bar w \leq w$. By definition of viscosity supersolution, we have
\[
(-L+b\cdot\nabla)\bar w(z) \geq -\eta.
\]

On the one hand, this implies
\[
(-L+b\cdot\nabla)(\bar w-\psi_{t_0})(z) \geq -C\eta,
\]
for some $C$ depending on $n$, the ellipticity constants, and $\|b\|$. On the other hand, we can evaluate $\bar w - \psi_{t_0}$ classically at $z$,
\begin{align*}
(-L+& b\cdot\nabla)(\bar w-\psi_{t_0})(z)  = -L(\bar w-\psi_{t_0})(z) \\
& \leq -\lambda \int_{\R^n} (\bar w - \psi_{t_0})(z+y)|y|^{-n-1} dy \leq -c(n) \lambda \int_{B_1\setminus B_{\delta}(z)} (\bar w - \psi_{t_0}) dy\\
&\leq -c(n)\lambda \int_{B_1} w^+ dy \leq -c(n)\lambda.
\end{align*}
We used here that $(\bar w - \psi_{t_0}) \chi_{B_1\setminus B_{\delta}(z)} \geq w^+$ in $B_1$.

In all, for $\eta$ small enough depending only on $n$, the ellipticity constants, and $\|b\|$, we reach a contradiction.
\end{proof}

With the previous lemma and the results from the previous section, we can now prove Proposition~\ref{prop.fblip}.
\begin{proof}[Proof of Proposition \ref{prop.fblip}]
Let $\delta > 0$ and $R_0$ to be chosen, and consider the rescaled function from Proposition~\ref{prop.regpt},
\[
v(x) = \frac{u(rx)}{r\|\nabla u\|_{L^\infty(B_r)}}.
\]

Thanks to Proposition~\ref{prop.regpt}, there exists some $e\in \Sp^{n-1}$ such that
\[
\left|\nabla v - (x\cdot e)^{\gamma(b\cdot e/\chi(e))}_+ e\right| \leq \delta\quad \textrm{in}\quad B_{R_0}.
\]
Recall $\gamma$ and $\chi$ are given by \eqref{eq.gamma}-\eqref{eq.chi}.

Now let $e'\in \Sp^{n-1}$ be such that (assuming $\ell \leq 1$)
\[
e'\cdot e \geq \frac{\ell}{\sqrt{1+\ell^2}} \geq \frac{\ell}{2}.
\]

Notice that
\[
\nabla v\cdot e' \geq \frac{\ell}{2} (x\cdot e)^{\gamma(b\cdot e/\chi(e))}_+ - \delta \quad\textrm{in}\quad B_{R_0},
\]
and
\[
\big|(-L+b\cdot\nabla)(\nabla v \cdot e')\big| \leq \delta \quad \textrm{in}\quad \{v > 0\}.
\]

Define
\[
w = \frac{C_1}{\ell} (\nabla v\cdot e')\chi_{B_2},
\]
for some $C_1$ such that
\[
\int_{B_1} w^+ \geq 1.
\]

Notice that, if $\delta$ is small enough, then $C_1$ depends only on $n$, $\ell$, $\|b\|$, and the ellipticity constants.

Let us call $E = \{v = 0\}$. If $R_0$ is large enough, depending only on $n$, $\ell$, $\varepsilon$, $\|b\|$, $\delta$, and the ellipticity constants, then $w$ satisfies
\begin{equation}
\label{eq.w}
\left\{\begin{array}{ll}
(-L+b\cdot \nabla)w \geq -\frac{CC_1}{\ell}\delta \geq -\eta & \quad \textrm{in}\quad B_1\setminus E\\
w = 0 & \quad \textrm{in}\quad E\cup(\R^n\setminus B_2) \\
w \geq -\frac{C_1}{\ell}\delta \geq -\eta & \quad \textrm{in}\quad B_2 \setminus E.\\
\end{array}\right.
\end{equation}
We are using here that, for $x\in B_1\setminus E$,
\begin{align*}
(-L+b\cdot \nabla)w(x) & \geq -\frac{C_1}{\ell} \delta- (-L+b\cdot \nabla)\left(\frac{C_1}{\ell} (\nabla v\cdot e')\chi_{B_2^c}\right)(x)\\
& \geq -\frac{C_1}{\ell} \delta + \frac{C_1}{\ell} L (\nabla v\cdot e')\chi_{B_2^c}(x)\\
& \geq -\frac{C_1}{\ell} \delta  + \lambda \frac{C_1}{\ell}\int_{B_{R_0-1}} \frac{(\nabla v\cdot e')\chi_{B_2^c}(x+y) + (\nabla v\cdot e')\chi_{B_2^c}(x-y)}{2|y|^{n+1}}\\
&~~~~~~~~~+ \lambda \frac{C_1}{\ell}\int_{B_{R_0-1}^c} \frac{(\nabla v\cdot e')\chi_{B_2^c}(x+y) + (\nabla v\cdot e')\chi_{B_2^c}(x-y)}{2|y|^{n+1}}\\
& \geq  -\frac{C_1}{\ell} \delta  - \lambda\frac{C_1}{\ell} \hat C\delta - \hat c \geq -\frac{CC_1}{\ell}\delta,
\end{align*}
where $R_0$ is chosen large enough so that $\hat c$ can be comparable to the other terms (which can be done, thanks to the fact that $\nabla v$ grows as $R^{1-\varepsilon}$). Notice that $C$ depends only on $\lambda$ and $n$.

In all, we can choose $\delta$ small enough so that
\[
\frac{CC_1}{\ell}\delta \leq \eta
\]
for the constant $\eta$ given in Lemma~\ref{lem.lipreg}.

Therefore, applying Lemma~\ref{lem.lipreg} to the function $w$ we get that
\[
w\geq 0\quad\textrm{in}\quad B_{1/2},
\]
or equivalently,
\[
\de_{e'} u \geq 0\quad\textrm{in}\quad B_{r/2},
\]
for all $e'\in \Sp^{n-1}$ such that $e'\cdot e \geq \frac{\ell}{\sqrt{1+\ell^2}}$. This implies that $\de\{u>0\}$ is Lipschitz in $B_r$, with Lipschitz constant smaller than $\ell$.
\end{proof}

Finally, combining Proposition~\ref{prop.fblip} with the boundary regularity result in Theorem~\ref{thm.bdharnack2} we show that the free boundary is $C^{1,\alpha}$ around regular points.

\begin{prop}
\label{prop.C1sigma}
Let $L$ be an operator of the form \eqref{eq.L}-\eqref{eq.L.cond}, and let $b\in\R^n$. Let $u$ be a solution to \eqref{eq.pb}-\eqref{eq.pb2}-\eqref{eq.pb3}. Assume that $x_0$ is a regular free boundary point.

Then, there exists $r > 0$ such that the free boundary is $C^{1,\alpha}$ in $B_r(x_0)$ for some $\alpha > 0$ depending only on $n$, $\|b\|$, and the ellipticity constants.
\end{prop}
\begin{proof}
Without loss of generality assume $x_0 = 0$ and that $\nu(0) = e_n$, where $\nu(0)$ denotes the normal vector to the free boundary at 0 pointing towards $\{u > 0\}$.

By Proposition~\ref{prop.fblip}, we already know the free boundary is Lipschitz around 0, with Lipschitz constant 1 in a ball $B_\rho$. Let $v_1 = \frac{1}{\sqrt{2}}\left(\de_i u + \de_n u\right)$ for any fixed $i\in\{1,\dots,n-1\}$, and let $v_2 = \de_n u$. We first show that for some $r > 0$ and $\alpha > 0$,
\begin{equation}
\label{eq.harnv1v2}
\left\|\frac{v_1}{v_2}\right\|_{C^\alpha\left(\{u > 0\}\cap B_r \right)} = \frac{1}{\sqrt{2}}\left\|1+\frac{\de_i u }{\de_n u}\right\|_{C^\alpha\left(\{u > 0\}\cap B_r \right)} \leq C.
\end{equation}

Define $w$ as in the proof of Proposition~\ref{prop.fblip}, i.e., $w = C_1 (\nabla v\cdot e')\chi_{B_2}$, where $v$ is the rescaling given by Proposition~\ref{prop.regpt}, and $e'$ is such that $e'\cdot e \geq \frac{\ell}{2}$ (choose $\ell = 1$ for example).

From the proof of Proposition~\ref{prop.fblip} we know that $w \geq 0 $ in $B_{1/2}$ (if, using the same notation, $R_0$ is large enough and $\delta$ is small enough; i.e., the rescaling defining $v$ is appropriately chosen). Now define
\[
\tilde{w} = C_1(\nabla v\cdot e')_+
\]
and notice that
\[
\big|(-L+b\cdot \nabla)\tilde{w}\big| \leq \eta  \quad \textrm{in}\quad B_{1/4}\setminus \{v = 0\}
\]
for some $\eta > 0$ that can be made arbitrarily small by choosing the appropriate (small) $\delta > 0$ and (large) $R_0$ in the rescaling given by Proposition~\ref{prop.regpt}. The previous inequality follows from the fact that $(\nabla v \cdot e')_- \leq \delta$ in $B_{R_0}$, $(\nabla v \cdot e')_- \leq 2 \left(1+ |x|^{1-\varepsilon}\right)$ in $B_{R_0^c}$, and $ (\nabla v \cdot e')_-\equiv 0$ in $B_{1/2}$.

Let $e_{in} := \frac{1}{\sqrt{2}}\left(e_i+e_n\right)$, and define $w_1= C\left(\nabla v\cdot e_{in}\right)_+$ and $w_2 = C(\nabla v\cdot e_n)_+$ (taking $e' = e_{in}$ and $e' = e_n$). Now notice that $w_1$ and $w_2$ fulfil the hypotheses of the boundary regularity result in Theorem~\ref{thm.bdharnack2}, and $w_1 = C(\nabla v\cdot e_{in})$ and $w_2 = C(\nabla v\cdot e_{n})$ in $B_{1/2}$. Thus, applying Theorem~\ref{thm.bdharnack2} to $w_1$ and $w_2$ we obtain that there exists some $\alpha > 0$ such that
\[
\left\|\frac{w_1}{w_2}\right\|_{C^\alpha(\{v > 0\} \cap B_{1/8})} \leq C.
\]
Going back to the rescalings defining $\tilde w$ we reach that for some $r> 0$, \eqref{eq.harnv1v2} holds.

Once we have \eqref{eq.harnv1v2} the procedure is standard. Notice that the components of the normal vector to the level sets $\{ u = t\}$ for $t > 0$ can be written as
\[
\nu^i(x) = \frac{\de_i u}{|\nabla u|}(x) = \frac{\de_i u/\de_n u}{\left(\sum_{j = 1}^{n-1} \left(\de_j u/\de_n u\right)^2 + 1\right)^{1/2}},
\]
\[
\nu^n(x) = \frac{\de_n u}{|\nabla u|}(x) = \frac{1}{\left(\sum_{j = 1}^{n-1} \left(\de_j u/\de_n u\right)^2 + 1\right)^{1/2}},
\]
for $u(x) = t>0$. In particular, from the regularity of $\de_i u / \de_n u$ given by \eqref{eq.harnv1v2}, we obtain $\nu$ is $C^\alpha$ on these level sets; that is, $|\nu(x) - \nu(y) | \leq C|x-y|^\alpha$ whenever $x, y \in \{u = t\}\cap B_r$. Now let $t\downarrow 0$ and we are done.
\end{proof}

\section{Estimates in $C^{1,\alpha}$ domains}
\label{sec.7}
Once we know that the free boundary is $C^{1,\alpha}$ around regular points, we need to find the expansion of the solution \eqref{eq.expansion} around such points. To do so, we establish fine boundary regularity estimates for solutions to elliptic problem with critical drift in arbitrary $C^{1,\alpha}$ domains. That is the aim of this section.

The main result of this section is the following, for the Dirichlet problem with the operator $-L+b\cdot\nabla$ in $C^{1,\alpha}$ domains. We will use it on the derivatives of the solution to the obstacle problem.

\begin{thm}
\label{thm.expansion_}
Let $L$ be an operator of the form \eqref{eq.L}-\eqref{eq.L.cond}, let $b\in \R^n$ and let $\Omega$ be a $C^{1,\alpha}$ domain.

Let $f\in L^\infty (\Omega\cap B_1)$, and suppose $u\in L^\infty(\R^n)$ satisfies
\begin{equation}
\left\{\begin{array}{rcll}
(-L+b\cdot\nabla )u& = & f & \quad \textrm{in}\quad \Omega\cap B_1\\
u& = &0 & \quad \textrm{in}\quad B_1\setminus \Omega. \\
\end{array}\right.
\end{equation}

Then, for each boundary point $x_0\in B_{1/2}\cap \de\Omega$, there exists a constant $Q$ with $|Q|\leq C\left(\|u\|_{L^\infty(\R^n)} + \|f\|_{L^\infty(\Omega\cap B_1)}\right)$ such that for all $x\in B_1$
\[
\left| u(x) - Q \big((x-x_0)\cdot \nu(x_0) \big)_+^{\tilde\gamma(x_0)} \right| \leq C\left(\|u\|_{L^\infty(\R^n)} + \|f\|_{L^\infty(\Omega\cap B_1)}\right)|x-x_0|^{\tilde\gamma(x_0)+\sigma},
\]
where $\sigma > 0$ and $\nu(x_0)$ is the normal unit vector to $\de\Omega$ at $x_0$ pointing towards the interior of $\Omega$, and $\tilde\gamma(x_0)$ is defined in \eqref{eq.tildegamma}. The constant $C$ depends only on $n$, $\alpha$, $\Omega$, the ellipticity constants, and $\|b\|$; and the constant $\sigma$ depends only on $n$, $\alpha$, the ellipticity constants, and $\|b\|$.
\end{thm}

To prove Theorem~\ref{thm.expansion_} we will need several ingredients.

\subsection{A supersolution and a subsolution}
In this section we denote
\[
d(x):= {\rm dist}(x, \R^n\setminus \Omega).
\]
We will also use the following.
\begin{defi}
\label{defi.rendist}
Given a $C^{1,\alpha}$ domain $\Omega$, we consider $\varrho$ a regularised distance function to $C^{1,\alpha}$; i.e., a function that satisfies
\[
\tilde{K}^{-1} d \leq \varrho \leq \tilde{K} d,
\]
\[
\|\varrho\|_{C^{1,\alpha}(\Omega)} \leq \tilde{K}\quad \textrm{and}\quad |D^2\varrho|\leq \tilde{K}d^{\alpha - 1},
\]
where the constant $\tilde{K}$ depends only on $\alpha$ and the domain $\Omega$.
\end{defi}

The existence of such regularised distance was discussed, for example, in \cite[Remark 2.2]{RS15}.

We next construct a supersolution, needed in our proof of Theorem~\ref{thm.expansion_}.

\begin{prop}[Supersolution]
\label{prop.supersolution}
Let $L$ be an operator of the form \eqref{eq.L}-\eqref{eq.L.cond}, and let $b\in \R^n$. Let $\Omega$ be a $C^{1,\alpha}$ domain for some $\alpha > 0$, and suppose $0\in \de\Omega$.

Let $\nu :\de\Omega\to \Sp^{n-1}$ be the outer normal vector at the points of the boundary of $\Omega$, let $\gamma$ be defined by \eqref{eq.gamma}, and $\chi$ by \eqref{eq.chi}. Let us also define
\[
\gamma_0 := \gamma\left(\frac{b\cdot \nu(0)}{\chi(\nu(0))}\right),
\]
and
\begin{equation}
\label{eq.etanu}
\eta_\nu := \inf \left\{\eta \geq 0 : \gamma\left(\frac{b\cdot \nu(x)}{\chi(\nu(x))}\right) \geq \gamma_0 - \eta\quad \forall x\in \de\Omega\cap B_1 \right\}.
\end{equation}

Let $\phi := \varrho^\kappa$ for a fixed $0<\kappa<\gamma_0-2\eta_\nu$, and where $\varrho$ is the regularised distance given by Definition~\ref{defi.rendist}. Then, there exist $\delta > 0$ and $\hat C>0$ such that
\begin{equation}
\left\{\begin{array}{rcll}
\hat C(-L+b\cdot\nabla )\phi & \geq& 1 & \quad \textrm{in}\quad B_{1/2}\cap \{x : 0<d(x) \leq \delta\}\\
\hat C\phi&\geq &1 & \quad \textrm{in}\quad B_{1/2}\cap \{x : d(x) \geq \delta\}. \\
\end{array}\right.
\end{equation}
The constants $\delta$ and $\hat C$ depend only on $n$, $\Omega$, $\kappa$, the ellipticity constants, and $\|b\|$.
\end{prop}
\begin{proof}
Pick any $x_0\in B_{1/2}\cap \{x: d(x) \leq \delta\}$, and define
\[
l_0 (x) = \big(\varrho(x_0) + \nabla \varrho (x_0)\cdot (x-x_0)\big)_+.
\]

Notice that, whenever $l_0 > 0$, if we define $\hat\varrho_0 := \frac{\nabla \varrho(x_0)}{|\nabla \varrho(x_0)|}$ and $z = \hat\varrho_0\cdot x$ then
\begin{align*}
(-L+b\cdot\nabla)l_0^\kappa(x) & = \left(\chi(\hat\varrho_0)(-\Delta)^{1/2} + \left(b \cdot \hat\varrho_0\right)\de\right)\big(|\nabla\varrho(x_0)|z + c_0\big)^{\kappa}_+ \\
& = |\varrho(x_0)| \chi(\hat\varrho_0) c\big(\kappa, b \cdot \hat\varrho_0/\chi(\hat\varrho_0)\big)\big(|\nabla\varrho(x_0)|z+c_0\big)^{\kappa-1}_+,
\end{align*}
where $c_0 = \varrho(x_0) -\nabla \varrho(x_0)\cdot x_0$, and $c(\kappa,  b \cdot\hat\varrho_0/\chi(\hat\varrho_0))$ is the constant arising from Proposition~\ref{prop.1DL}. We want to check that this constant is positive, which is equivalent to saying (again, from Proposition~\ref{prop.1DL}) that
\[
\kappa < \gamma\left(\frac{b\cdot \hat\varrho_0}{\chi(\hat\varrho_0)}\right).
\]

To see this, it is enough to check that
\[
\gamma_0 - 2\eta_\nu\leq  \gamma\left(\frac{b\cdot \hat\varrho_0}{\chi(\hat\varrho_0)}\right),
\]
which will be true for some small $\delta > 0$ and for any $x_0 \in B_{1/2}\cap \{x: d(x) \leq \delta\} $ if
\[
\lim_{\delta \downarrow 0} \inf_{\substack{y\in B_{1/2} \\ 0<d(y) \leq \delta}}\sup_{x\in \de\Omega\cap B_{3/4}} \frac{\nabla \varrho(y)}{|\nabla \varrho(y)|}\cdot \nu(x) = 1,
\]
i.e., $\nabla \varrho$ normalised is close to some unit normal vector to the boundary as $\delta$ goes to zero (notice that $\gamma$ and $\chi$ are continuous). But this is true since $\varrho$ is a $C^{1,\alpha}$ function, so in particular, its gradient is continuous, and the boundary is a level set of $\varrho$; i.e., $\nabla \varrho(y) = |\nabla \varrho (y)|\nu(y)$ for any $y$ on the boundary. It is important to remark that the modulus of continuity of $\nabla \varrho$ depends only on $\Omega$.

Now notice that
\begin{equation}
\label{eq.eqx0}
l_0 (x_0) = \varrho(x_0) \quad \quad \nabla l_0 (x_0) = \nabla \varrho (x_0).
\end{equation}
Let $\tilde{\varrho}$ be a $C^{1,\alpha}(\R^n)$ extension of $\varrho$ to the whole $\R^n$ with $\varrho \leq 0$ in $\R^n \setminus \Omega$. Then we have
\[
\big|\varrho(x_0) +\nabla \varrho(x_0)\cdot y - \tilde{\varrho}(x_0 + y)\big|\leq C|y|^{1+\alpha}.
\]
By using that $|a_+ - b_+| \leq |a-b|$ we find
\[
\big|l_0(x_0 + y) - \varrho(x_0 + y)\big|\leq C|y|^{1+\alpha}.
\]

Now, also using that $|a^t - b^t| \leq |a-b|(a^{t-1} + b^{t-1})$ for $a,b \geq 0$, $|a^t - b^t| \leq C|a-b|^t$, and saying $d_0 = d(x_0)$ we get
\begin{equation}
\label{eq.supereq}
|\phi - l_0^\kappa|(x_0+y) \leq \left\{\begin{array}{ll}
 Cd_0^{\kappa-1}|y|^{1+\alpha} & \quad \textrm{for}\quad y\in B_{d_0/(\tilde{K}+1)}\\
 C|y|^{(1+\alpha)\,\kappa}& \quad \textrm{for}\quad y\in B_{1}\setminus B_{d_0/(\tilde{K}+1)} \\
 C|y|^{\kappa} &\quad \textrm{for}\quad y\in \R^n\setminus B_{1}.\\
\end{array}\right.
\end{equation}
We have used here that, in $B_{d_0/(\tilde{K}+1)}$, $l_0^{\kappa-1} \leq Cd_0^{\kappa-1}$ and $\varrho^{\kappa-1} \leq Cd_0^{\kappa-1}$. Here, $\tilde{K}$ denotes the constant given in Definition~\ref{defi.rendist}.
Putting all together
\begin{align*}
(-L+& b\cdot\nabla)\phi(x_0)  =\\
&= (-L+ b\cdot\nabla)(\phi-l_0^\kappa)(x_0)+ (-L+b\cdot\nabla)l_0^\kappa(x_0)\\
& \geq L (l_0^\kappa-\phi)(x_0) + c(\kappa) d_0^{\kappa -1}\\
& = \int_{\Sp^{n-1}} \int_{0}^\infty \left( (l_0^\kappa-\phi)(x_0+ r\theta) + (l_0^\kappa-\phi)(x_0 - r\theta) \right)\frac{dr}{r^{2}}d\mu(\theta) + c(\kappa) d_0^{\kappa -1}\\
& \geq -C\left(\int_0^{d_0/(\tilde{K}+1)} \frac{d_0^{\kappa-1}r^{1+\alpha}}{r^2}dr + \int_{d_0/(\tilde K +1)}^{1} \frac{r^{(1+\alpha)\,\kappa}}{r^2}dr + \int_{1}^{\infty} \frac{r^{\kappa}}{r^2}dr \right) +c(\kappa) \rho^{\kappa -1}\\
& \geq -Cd_0^{\kappa - 1+\alpha}-Cd_0^{(1+\alpha)\,\kappa - 1}+c(\kappa) d_0^{\kappa -1}.
\end{align*}

Notice that the right-hand side tends to $+\infty$ as $\delta \downarrow 0$ independently of the $x_0$ chosen. Thus, we can choose $\delta$ small enough so that the right-hand side is greater than 1. Then, by choosing $\hat C \geq 1$ such that $\hat C \phi\geq 1$ in $B_{1/2}\cap \{x: d(x) > \delta\}$ we are done.
\end{proof}

We can similarly find a subsolution for the problem. It will be used in the next section.

\begin{lem}[Subsolution]
\label{lem.subsolution}
Let $L$ be an operator of the form \eqref{eq.L}-\eqref{eq.L.cond}, and let $b\in\R^n$. Let $\Omega$ be a $C^{1,\alpha}$ domain for some $\alpha > 0$, and suppose $0\in \de\Omega$.

Let $\nu :\de\Omega\to \Sp^{n-1}$ be the outer normal vector at the points of the boundary of $\Omega$, let $\gamma$ be defined by \eqref{eq.gamma}, and $\chi$ by \eqref{eq.chi}. Let us also define
\[
\gamma_0 := \gamma\left(\frac{b\cdot \nu(0)}{\chi(\nu(0))}\right),
\]
and
\begin{equation}
\label{eq.etanu_2}
\eta_\nu^{(2)} := \inf \left\{\eta \geq 0 : \gamma\left(\frac{b\cdot \nu(x)}{\chi(\nu(x))}\right) \leq \gamma_0 + \eta\quad \forall x\in \de\Omega\cap B_1 \right\}.
\end{equation}

Let $\phi := \varrho^{\kappa_2}$ for any fixed $1>\kappa_2>\gamma_0+2\eta_\nu^{(2)}$. Then, there exist $\delta > 0$ and $\hat C>0$ such that
\begin{equation}
\left\{\begin{array}{rcll}
(-L+b\cdot\nabla )\phi & \leq& -1 & \quad \textrm{in}\quad B_{1/2}\cap \{x : 0<d(x) \leq \delta\}\\
\phi&\leq &\hat C & \quad \textrm{in}\quad B_{1/2}\cap \{x : d(x) > \delta\}. \\
\end{array}\right.
\end{equation}
The constants $\delta$ and $\hat C$ depend only on $n$, $\Omega$, $\kappa_2$, the ellipticity constants, and $\|b\|$.
\end{lem}
\begin{proof}
The proof follows by the same steps as the proof of Proposition~\ref{prop.supersolution}. Using the same notation, one just needs to notice that when evaluating
\[
(-L+b\cdot\nabla)l_0^{\kappa_2}(x) = c\big(\kappa_2, b \cdot \hat\varrho_0/\chi(\hat\varrho_0)\big)\big(|\nabla\varrho(x_0)|z+c_0\big)^{\kappa_2-1}_+,
\]
now the constant $c(\kappa_2)$ is negative (independently of the $\kappa_2$ chosen, as before). Thus,
\[
(-L+ b\cdot\nabla)\phi(x_0)\leq Cd_0^{\kappa_2 - 1+\alpha}+Cd_0^{(1+\alpha)\,\kappa_2 - 1}+c(\kappa) d_0^{\kappa_2 -1},
\]
for negative $c(\kappa_2)$, so that if $d_0$ is small enough we obtain the desired result.
\end{proof}

\subsection{Hölder continuity up to the boundary in $C^{1,\alpha}$ domains}

The aim of this subsection is to prove Proposition~\ref{prop.boundestimatesg} below. Before doing that, let us introduce a definition.

\begin{defi}
\label{defi.split}
We say that $\Gamma\subset\R^n$ is a $C^{1,\alpha}$ graph splitting $B_1$ into $U^+$ and $U^-$ if there exists some $f_\Gamma\in C^{1,\alpha}(\R^{n-1})$ such that
\begin{enumerate}[$\bullet~~$]
\item $\Gamma := \{(x', f_\gamma(x'))\cap B_1 \textrm{ for } x'\in \R^{n-1}\}$;
\item $U^+ := \{(x',x_n)\in B_1 : x_n > f_\Gamma(x')\}$;
\item $U^- := \{(x',x_n)\in B_1 : x_n < f_\Gamma(x')\}$.
\end{enumerate}
Under these circumstances, we refer to the $C^{1,\alpha}$ norm of $\Gamma$ as $\|f_\Gamma\|_{C^{1,\alpha}(D')}$, where $D':= \{x'\in \R^n : (x',f_\Gamma(x'))\in B_1\}$.
\end{defi}

\begin{prop}
\label{prop.boundestimatesg}
Let $L$ be an operator of the form \eqref{eq.L}-\eqref{eq.L.cond}, and let $b\in \R^n$. Let $\Gamma$ be a $C^{1,\alpha}$ graph splitting $B_1$ into $U^+$ and $U^-$, according to Definition~\ref{defi.split}, and suppose $0\in \Gamma$.

Let $f\in L^\infty (U^+)$, let $g\in C^\beta(\overline{U_-})$, and suppose $u\in C(\overline{B_1})$ satisfying the growth condition $|u(x) |\leq M (1+|x|)^\Upsilon$ in $\R^n$ for some $\Upsilon < 1$. Assume also that $u$ satisfies in the viscosity sense
\begin{equation}
\left\{\begin{array}{rcll}
(-L+b\cdot\nabla )u& = & f & \quad \textrm{in}\quad U^+\\
u& = &g & \quad \textrm{in}\quad U^-. \\
\end{array}\right.
\end{equation}

Then there exists some $\sigma > 0$ such that $u\in C^\sigma(\overline{B_{1/2}})$ with
\[
\|u\|_{C^\sigma(B_{1/2})} \leq C\big(\|u\|_{L^\infty(B_1)}  + \|g\|_{C^\beta(U^-)} + \|f\|_{L^\infty(U^+)} + M\big).
\]
The constants $C$ and $\sigma$ depend only on $n$, $\alpha$, the $C^{1,\alpha}$ norm of $\Gamma$, $\Upsilon$, the ellipticity constants, and $\|b\|$.
\end{prop}
\begin{proof}
Let $\tilde{u} = u \chi_{B_1}$ so that $(-L+b\cdot \nabla)\tilde{u} = f + L (u\chi_{B_1^c}) =: \tilde{f}$ in $U^+\cap B_{3/4}$, and $\tilde{u} = g$ in $U^-$. Note that $\|\tilde{f}\|_{L^\infty(U^+ \cap B_{3/4})} \leq C(\|f\|_{L^\infty(U^+)} + M) =: C_0$ for some constant $C$ depending only on $n$, $\Upsilon$, and the ellipticity constants.

We begin by proving that for some small $\epsilon > 0$, and for some $C$, we have
\begin{equation}
\label{eq.utildg}
\|\tilde{u} - g(z)\|_{L^\infty(B_r(z))} \leq Cr^\epsilon\quad \textrm{for all}\quad r\in (0,1),\quad\textrm{and for all}\quad z\in \Gamma\cap B_{1/2},
\end{equation}
where $\epsilon> 0$ and $C$ depend only on $n$, $C_0$, $\|u\|_{L^\infty(B_1)}$, $\|g\|_{C^\beta(U^-)}$, the ellipticity constants, and $\|b\|$.

Let us define a $C^{1,\alpha}$ domain that will be used in this proof, analogous to a fixed ball if the surface $\Gamma$ was $C^{1,1}$.

Thus, we define $P$ as a fixed $C^{1,\alpha}$ bounded convex domain with diameter 1 that coincides with $\{x = (x_1,\dots,x_n)\in \R^n: x_n \geq |(x_1,\dots,x_{n-1})|^{1+\alpha} \}$ in $B_{1/2}$. Let $y_P$ be a fixed point inside the domain, which will be treated as the \emph{center}. Let us call $P_R$ the rescaled version of such domain with diameter $R$ and \emph{center} $y_{P_R}$, and let us define
\[
P_R^{(\delta)} := \{x\in \R^n : {\rm dist}(x, P_R) \leq\delta\}.
\]
As an abuse of notation we will also call $P_R$ any rotated and translated version that will be given by the context.

Note that, since $\Gamma$ is $C^{1,\alpha}$, there exists some $\rho_0 \in (0,1)$ depending on the $C^{1,\alpha}$ norm of $\Gamma$ such that any point $z \in \Gamma\cap B_{1/2}$ can be touched by some $P_{\rho_0}$ rotated and translated correspondingly and contained completely in $U^-$.

Let us now consider the supersolution given by Proposition~\ref{prop.supersolution} with respect to the domain $\R^n\setminus P$.

That is, there is some function $\phi_P$ such that, for some constants $\delta > 0$ and $C$ fixed,
\begin{equation}
\label{eq.phiP}
\left\{\begin{array}{rcll}
 (-L+b\cdot\nabla )\phi_P & \geq& 1 & \quad \textrm{in}\quad P^{(\delta)}\setminus P\\
\phi_P &\geq &1 & \quad \textrm{in}\quad \R^n \setminus P^{(\delta)} \\
\phi_P &= &0& \quad \textrm{in}\quad P \\
\phi_P &\leq &Cd^{\kappa}& \quad \textrm{in}\quad \R^n,\\
\end{array}\right.
\end{equation}
where $d = {\rm dist} (x, P)$ and $0<\kappa <\min\left\{\gamma\left(\frac{b'\cdot e}{\chi(e)}\right) : \|b'\| = \|b\|, e\in \Sp^{n-1}\right\}$ can also be fixed --- recall that $\gamma$ and $\chi$ are given by \eqref{eq.gamma}-\eqref{eq.chi}.

\begin{figure}
\centering
\includegraphics{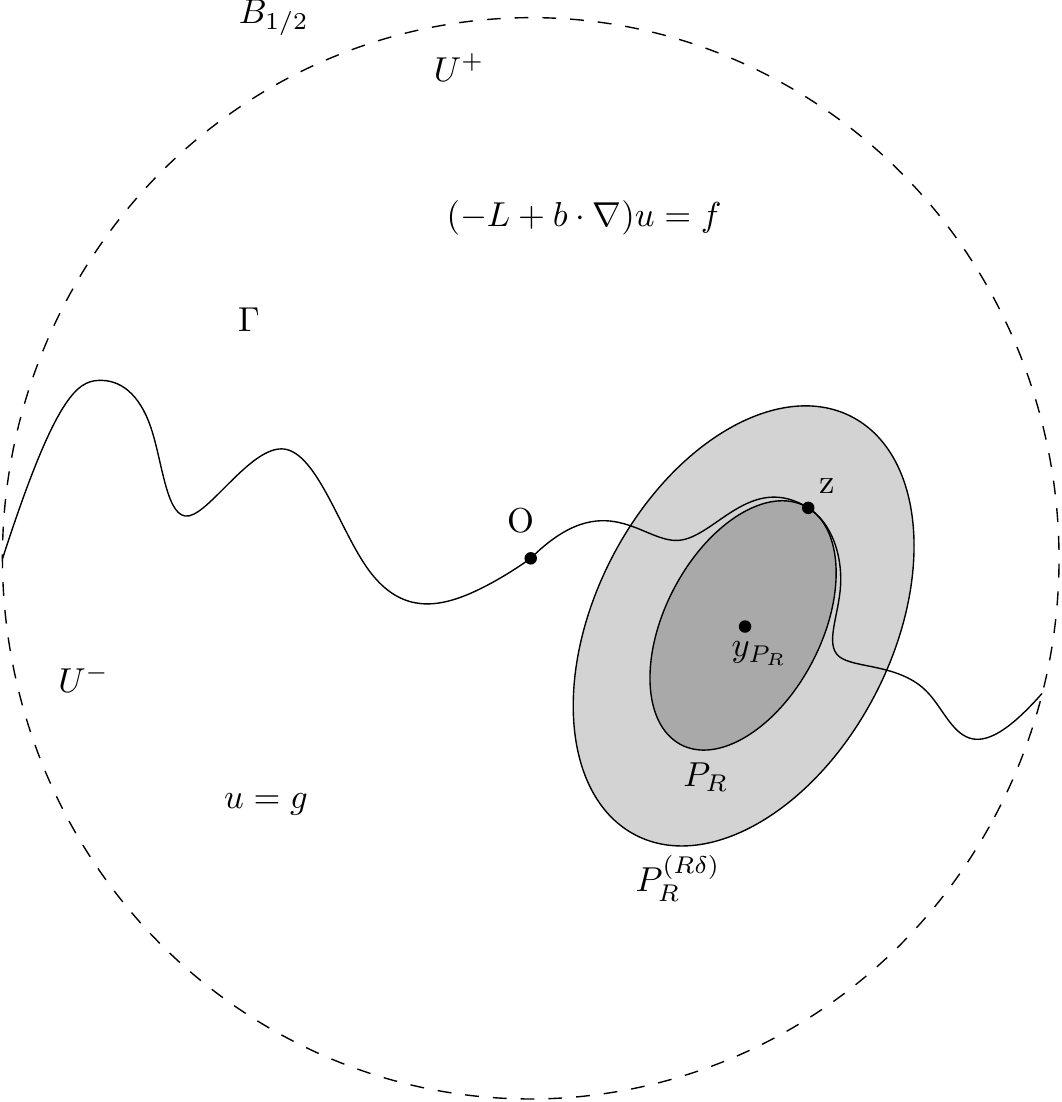}
\caption{}
\label{fig.draw}
\end{figure}

Let $P'$ be a rotated version of $P$, and let $\phi_{P'}$ be the corresponding rotated supersolution. Notice that we can assume that $\phi_{P'}$ also fulfils \eqref{eq.phiP} (with $P'$ instead of $P$), since while the operator $(-L+b\cdot \nabla)$ is not rotation invariant, only an extra positive constant arises depending on the ellipticity constants and $\|b\|$.

Given a rotated, scaled and translated version of the domain $P$, $P_R$, we will denote the corresponding supersolution (the rotated, scaled and translated version of $\phi_P$) by $\phi_{P_R}$.

Let now $z\in \Gamma\cap B_{1/2}$. For any $R\in (0,\rho_0)$ there exists some rescaled, rotated and translated domain $P_R\subset U^-$ touching $\Gamma$ at $z$. Recall that $y_{P_R}$ is the \emph{center} of the domain $P_R$, so that in particular $|z-y_{P_R}| = C_P R$ for some constant $C_P$ that only depends on the domain $P$ chosen ($C_P\in (0,1)$ because the domain $P_R$ has diameter $R$). See Figure~\ref{fig.draw} for a representation of this situation.

Recall that $\phi_{P_R}$ is the supersolution corresponding to the domain $P_R$, with the $\delta$ given by Proposition~\ref{prop.supersolution} (now, when rescaling, $\delta$ becomes $R\delta$). Define the function
\[
\psi(x) = g(y_{P_R}) + \|g\|_{C^\beta(U^-)} \big((1+\delta)R\big)^\beta + \big(C_0 + \|u\|_{L^\infty(B_1)}\big) \phi_{P_R}.
\]

Note that $\psi$ is above $\tilde{u}$ in $U^-\cap P^{(R\delta)}_R$, since $\tilde{u} = g$ there and the distance from $y_{P_R}$ to any other point in $P^{(R\delta)}_R$ is at most $(1+\delta)R$.

On the other hand, in $P^{(R\delta)}_R\setminus P_R$ we have $(-L+b\cdot \nabla)\psi \geq (C_0 + \|u\|_{L^\infty(B_1)}) R^{-1} \geq C_0 \geq (-L + b\cdot\nabla) \tilde{u}$ since $R\leq \rho_0 < 1$; and outside $P^{(R\delta)}_R$ we have $\tilde{u} \leq \psi$. In all, $\tilde u \leq \psi$ everywhere by the maximum principle, and thus for any $ R\in (0,\rho_0)$
\[
\tilde u (x) - g(z) \leq C\big(R^{\beta} + (r/R)^\kappa\big)\quad \textrm{for all} \quad x\in B_r(z)\quad\textrm{and for all}\quad r\in (0,R\delta),
\]
for some constant $C$ that depends only on $n$, $C_0$, $\|u\|_{L^\infty(B_1)}$, $\|g\|_{C^\beta(U^-)}$, the ellipticity constants, and $\|b\|$. If $R$ is small enough we can take $r = R^2$, and repeat this reasoning upside down to get that
\[
\|\tilde{u} - g(z)\|_{L^\infty(B_r(z))} \leq C\left(r^{\beta/2} + r^{\kappa/2}\right) \leq Cr^\epsilon\quad\textrm{for all } \quad r\in (0,\delta^{2}),
\]
for $\epsilon = \min\left\{\frac{\beta}{2},\frac{\kappa}{2}\right\}$.
This yields the result \eqref{eq.utildg} by taking a larger $C$ if necessary.

Now let $x, y\in B_{1/2}$, and let $r = |x-y|$. We will show
\[
|u(x)-u(y)| \leq Cr^\sigma,
\]
for some $\sigma > 0$. If $x, y \in U^-$ we are done by the regularity of $g$. If $x\in U^+$, $y\in U^-$, we can take $z$ in the segment between $x$ and $y$, on the boundary $\Gamma$, and compare $x$ and $y$ to $z$, so that it is enough to consider $x, y\in U^+$.

Let $R = {\rm dist}(x, \Gamma) \geq {\rm dist}(y, \Gamma)$, and suppose $x_0,y_0\in \Gamma$ are such that ${\rm dist}(x, \Gamma) = {\rm dist}(x, x_0)$ and ${\rm dist}(y, \Gamma) = {\rm dist}(y, y_0)$. By interior estimates for the problem (see Proposition~\ref{prop.intest}),
\begin{equation}
\label{eq.intest2}
[u]_{C^\epsilon(B_{R/2}(x))} \leq CR^{-\epsilon}.
\end{equation}

Let $r < 1$, and let us separate two different cases
\begin{enumerate}[$~~~~\bullet~$]
\item Suppose $r \geq R^2/2$. Then, using \eqref{eq.utildg} and the regularity of $g$ we obtain
\begin{align*}
|u(x)-u(y)| & \leq |u(x)-u(x_0)|+|u(x_0)-u(y_0)|+|u(y_0)-u(y)|\\
& \leq CR^{\epsilon} + C(2R+r)^\beta \\
& \leq C(r^{\epsilon/2} + r^{\beta/2}) \leq Cr^{\epsilon/2}.
\end{align*}
\item Assume $r \leq R^2/2$, so that $y\in B_{R/2}(x)$. Thus, using \eqref{eq.intest2},
\[
|u(x)-u(y)|\leq CR^{-\epsilon} r^\epsilon \leq Cr^{\epsilon/2}.
\]
\end{enumerate}

In all, we have found $u\in C^{\sigma}(B_{1/2})$ for $\sigma = \epsilon/2$.
\end{proof}

\begin{rem}
When $U$ is $C^\infty$, the above Hölder estimate follows from the results in \cite{S94}, \cite{CD01}. We thank G. Grubb for pointing this out to us.
\end{rem}

\subsection{A Liouville theorem}
We next prove a Liouville-type theorem in the half-space for non-local operators with critical drift, that will be used to prove Theorem~\ref{thm.expansion_}.

\begin{thm}
\label{thm.liouv}
Let $L$ be an operator of the form \eqref{eq.L}-\eqref{eq.L.cond}, and let $b\in \R^n$. Let $u$ be any weak solution to
\begin{equation}
\left\{\begin{array}{rcll}
(-L+b\cdot\nabla )u & =& 0& \quad \textrm{in}\quad \R^n_+\\
u & = & 0 & \quad \textrm{in}\quad \R^n_-. \\
\end{array}\right.
\end{equation}

Assume also that for some $\varepsilon > 0$ and some constant $C$, $u$ satisfies
\[
\|u\|_{L^\infty(B_R)} \leq C R^{1-\varepsilon}\quad\textrm{for all}\quad R \geq 1.
\]

Then,
\begin{equation}
u(x) = C(x_n)_+^{\gamma(b_n/\chi)},
\end{equation}
for some $C > 0$, and where $b_n$ is the $n$-th component of $b$. The constant $\chi$ is defined by $\chi = \chi(e_n)$ where $\chi(e)$ is given by \eqref{eq.chi}, and $\gamma$ is given by \eqref{eq.gamma}.
\end{thm}

Before proving the Liouville theorem, let us prove it in the 1-dimensional case.

Notice that from Proposition~\ref{prop.bdharnack} it already follows that any non-negative solution must be either $u\equiv 0$ or the one found in Proposition~\ref{prop.1DL}. Here, however, we need the same result for solutions that may change sign.

\begin{prop}
\label{prop.1DL_2}
Let $b \in \R$, and let $u\in C(\R)$ be a function satisfying
\[
(-\Delta)^{1/2} u + bu' = 0\quad \textrm{in}\quad \R_+,\quad \quad u \equiv 0\quad \textrm{in}\quad \R_-,
\]
and $|u(x)| \leq C (1+|x|^{1-\varepsilon})$ for some $\varepsilon>0$. Then,
\[
u (x) = C_0(x_+)^{\gamma(b)},
\]
where $\gamma$ is given by \eqref{eq.gamma}.
\end{prop}
\begin{proof}
We first claim that
\begin{equation}
\label{eq.claimb}
 \left\|u/(x_+)^{\gamma(b)}\right\|_{C^{\sigma}([0,1])}\leq C
\end{equation}
for some $\sigma > 0$.

Indeed, let
\[
w = \chi_{[0,2]} u + \kappa \chi_{[3/2,2]},
\]
and recall that, for some $\hat C$,
\[
\|u\|_{L^\infty([0,R])} \leq \hat C R^{1-\varepsilon}.
\]

Notice that $w(0) = 0$, and that $w \leq C_0 (x)_+^{\gamma(b)}$ for $x \geq 1$, if $C_0$ is big enough depending only on $\kappa$ and $\hat C$. Choose $\kappa$ so that $(-\Delta)^{1/2} w \leq 0$ in $[0,1]$ so that by the maximum principle $ u = w \leq C_0 (x)_+^{\gamma(b)}$ in $[0,1]$. Doing the same for $-u$ we reach that
\[
|u| \leq C_0 (x)_+^{\gamma(b)}\quad \textrm{for} \quad x\in [0,1].
\]

Define now $\tilde u = u \chi_{(0,m)}+ M(x_+)^{\gamma(b)}$, where $M = M(m)$ is such that $\tilde u \geq 0$ in $(0,m)$. Notice that $\tilde u$ solves an equation of the form $(-\Delta)^{1/2} \tilde u + b\tilde u' = f_m(x)$ in $(0,1)$ for some bounded $f_m$ with $\|f_m\|_{L^\infty(0,1)} \downarrow 0$ as $m \to \infty$. We can now apply Theorem~\ref{thm.bdharnack2} with $\tilde u$ and $(x_+)^{\gamma(b)}$ to get that for some large enough $m$,
\[
\left\|\tilde u/(x_+)^{\gamma(b)}\right\|_{C^{\sigma}([0,1])}\leq C,
\]
for some $\sigma > 0$. Thus, we get \eqref{eq.claimb}.

Define $v = u - k(x_+)^{\gamma(b)}$, where $k = \lim_{x \downarrow 0} \frac{u(x)}{(x_+)^{\gamma(b)}}$. Then we have
\begin{equation}
\label{eq.vnearinf}
|v(x)| \leq C|x|^{1-\varepsilon}\quad \textrm{for} \quad x \geq 1,
\end{equation}
\begin{equation}
\label{eq.vnear0}
|v(x)| \leq C|x|^{\gamma(b) + \sigma}\quad \textrm{for} \quad x \in [0,2],
\end{equation}
and we can assume, without loss of generality, that $1-\varepsilon > \gamma(b)+\sigma$. Combining this with the interior estimates from Proposition~\ref{prop.intest} we obtain $v\in C^{\gamma(b) + \sigma}([0,1])$. Indeed, take $x, y \in [0,1]$, $x < y$. Let $r = y-x$ and $R = |y|$. Now separate two cases
\begin{enumerate}[$~~~\bullet~$]
\item If $2r \geq R$, by \eqref{eq.vnear0}
\begin{align*}
|v(x)-v(y)| & \leq |v(x)| + |v(y)| \leq C(|x|^{\gamma(b)+\sigma} + |y|^{\gamma(b)+\sigma})\\
& \leq C \big((R-r)^{\gamma(b)+\sigma}+R^{\gamma(b)+\sigma}\big) \leq C r^{\gamma(b)+\sigma}.
\end{align*}
\item If $2r < R$, then $x, y \in (y-R/2, y+R/2)$. By rescaling the estimates from Proposition~\ref{prop.intest} and using \eqref{eq.vnearinf}
\[
R^{\gamma(b)+\sigma}[v]_{C^{\gamma(b)+\sigma}\left(y-\frac{R}{2}, y+\frac{R}{2}\right)}  \leq C\left(\|v\|_{L^\infty\left(y-R, y+R\right) }+R^{1-\varepsilon}\right).
\]

Now, from \eqref{eq.vnear0}
\[
\|v\|_{L^\infty\left(y-R, y+R\right) }\leq CR^{\gamma(b)+\sigma},
\]
so that
\[
[v]_{C^{\gamma(b)+\sigma}\left(y-\frac{R}{2}, y+\frac{R}{2}\right)}  \leq C.
\]
\end{enumerate}
This implies
\[
\|v\|_{C^{\gamma(b)+\sigma}([0,1])} \leq C,
\]
as desired.

Now, we claim that using the interior estimates from Proposition~\ref{prop.intest} we obtain
\begin{equation}
\label{eq.vprime1}
|v'(x)|\leq C|x|^{-\varepsilon} \quad \textrm{for} \quad x\geq 1,
\end{equation}
and
\begin{equation}
\label{eq.vprime2}
|v'(x)| \leq C|x|^{\gamma(b) + \sigma- 1}  \quad \textrm{for} \quad x\in[0,1].
\end{equation}

Let us show that these last inequalities hold. The first one, \eqref{eq.vprime1}, follows using that $|v(x)| \leq C(1 + |x|^{1-\varepsilon})$, and that \eqref{eq.vnearinf}-\eqref{eq.vnear0} combined with the rescaled interior estimates in Proposition~\ref{prop.intest} yield
\begin{equation}
\label{eq.vprime0}
[v]_{C^{\gamma(b)+\sigma}(R, 2R)} \leq CR^{1-\varepsilon-\gamma(b)-\sigma}\quad\textrm{for}\quad R \geq 1.
\end{equation}

Indeed, take $0<\alpha < \gamma(b)+\sigma$, and any $h\in \R$ with $|h| \leq R/2$. Then by interior estimates applied to the incremental quotients,
\[
\left[\frac{v(x+h)-v(x)}{|h|^{\gamma(b)+\sigma}}\right]_{C^{1-\alpha}(R,2R)} \leq CR^{\alpha-\varepsilon-\gamma(b)-\sigma} \quad\textrm{for}\quad R \geq 1,
\]
with $C$ independent of the $h$ chosen. In particular, this yields
\[
[v']_{C^{\gamma(b)+\sigma-\alpha}(R,2R)} \leq CR^{\alpha-\varepsilon-\gamma(b)-\sigma}\quad\textrm{for}\quad R \geq 1.
\]
The inequality in \eqref{eq.vprime1} now follows comparing the value of $v'(2^k)$ for any $k\in \N$ with $v'(1)$ dyadically.

For the second inequality, \eqref{eq.vprime2}, we proceed similarly. Take $0<\alpha<  \gamma(b)+\sigma$, and for any $R > 0$ fixed take $|h|\leq R/2$ and notice that
\begin{equation}
\label{eq.vprime3}
\left[\frac{v(x+h)-v(x)}{|h|^{\gamma(b)+\sigma}}\right]_{C^{1-\alpha}(R,2R)} \leq CR^{\alpha-1}\quad\textrm{for}\quad 0<R<1,
\end{equation}
with $C$ independent of $h$. This follows from the interior estimates in Proposition~\ref{prop.intest} and the growth of $\frac{v(x+h)-v(x)}{|h|^{\gamma(b)+\sigma}}$ given by \eqref{eq.vprime0}. As before, this implies
\[
[v']_{C^{\gamma(b)+\sigma-\alpha}(R,2R)} \leq CR^{\alpha-1}\quad\textrm{for}\quad 0<R<1.
\]
Finally, the inequality \eqref{eq.vprime2} follows comparing the value of $v'(2^{-k})$ with $v'(1)$ dyadically. Thus, \eqref{eq.vprime1} and \eqref{eq.vprime2} are proved.

Define now the function
\[
\psi_A(x) = A\left((x_+)^{\gamma(b)} + (x_+)^{\gamma(b) - 1}\right),
\]
and notice that $\psi_A$ and $v'$ solve
\begin{equation}
\label{eq.psia1}
(-\Delta)^{1/2} \psi_A + b\psi_A' = 0\quad \textrm{in} \quad x> 0,
\end{equation}
\begin{equation}
\label{eq.psia2}
(-\Delta)^{1/2} v' + b(v')' = 0\quad \textrm{in} \quad x> 0.
\end{equation}

We have that $\psi_A > v'$ in $\{ x > 0\}$ for some large enough $A$, thanks to the growth of $v'$ in \eqref{eq.vprime1}-\eqref{eq.vprime2}. Choose the smallest nonnegative $A$ such that $\psi_A\geq v'$. Then, by the growth at zero and infinity of both $v'$ and $\psi_A$ they touch at some point in $(0,\infty)$. Moreover, if $A>0$, then we must have $\psi_A \not\equiv v'$.

Let $x_0 > 0$ be a point where $\psi_A(x_0) = v'(x_0)$. Notice that $\psi_A - v'$ is a non-negative (and non-zero) function with a minimum at $x_0$. Thus,
\[
\big((-\Delta)^{1/2} (\psi_A - v') + b(\psi_A - v')'\big)(x_0) = (-\Delta)^{1/2} (\psi_A - v')(x_0) < 0,
\]
which contradicts the fact that both $\psi_A$ and $v'$ are solutions to the problem, \eqref{eq.psia1}-\eqref{eq.psia2}. Thus, there is no positive $A$ such that $\psi_A$ and $v'$ touch at at least one point, so we must have $v' \leq 0$. Doing the same from below we reach $v' \geq 0$, and therefore $v'\equiv 0$. Hence, since $u(0) = 0$ we find $v \equiv 0$. In particular, this means that
\[
u = k(x_+)^{\gamma(b)},
\]
as desired.
\end{proof}

We can now prove the Liouville theorem.

\begin{proof}[Proof of Theorem~\ref{thm.liouv}]
Let us first see that the solution is 1-dimensional in the direction $e_n$.

Given $\rho \geq 1$, define
\[
v_\rho(x) = \rho^{-\varepsilon+1}u(\rho x).
\]

Notice that
\[
\|v_\rho\|_{L^\infty(B_R)} = \rho^{-\varepsilon+1}\|u(\rho\cdot)\|_{L^\infty(B_R)} = \rho^{-\varepsilon+1}\|u\|_{L^\infty(B_{\rho R})} \leq CR^{1-\varepsilon}.
\]
Moreover, by the homogeneity of $(-L+b\cdot\nabla)$,
\begin{equation}
\left\{\begin{array}{rcll}
(-L+b\cdot\nabla )v_\rho & =& 0& \quad \textrm{in}\quad \R^n_+\\
v_\rho & = & 0 & \quad \textrm{in}\quad \R^n_-. \\
\end{array}\right.
\end{equation}

Define now $\tilde{v}_\rho = v_{\rho}\chi_{B_2}$, so that $\tilde{v}_\rho \in L^\infty(\R^n)$. We now have
\begin{equation}
\left\{\begin{array}{rcll}
(-L+b\cdot\nabla )\tilde v_\rho & =& g_\rho& \quad \textrm{in}\quad B_1^+\\
\tilde v_\rho & = & 0 & \quad \textrm{in}\quad B_1^-, \\
\end{array}\right.
\end{equation}
for some $g_\rho$ with $\|g_\rho\|_{L^\infty(B_1^+)}\leq C_0$ with $C_0$ independent of $\rho$. Indeed,
\[
(-L+b\cdot\nabla )\tilde v_\rho = (-L+b\cdot\nabla )( v_\rho - v_\rho\chi_{B_2^c}) = L (v_\rho\chi_{B_2^c}) \leq C_0\quad \textrm{in}\quad B_1^+,
\]
where the last inequality follows thanks to the uniform growth control on $v_\rho$.

Now, by Proposition~\ref{prop.boundestimatesg},
\[
\|v_\rho\|_{C^\sigma(B_{1/2})} = \|\tilde{v}_\rho\|_{C^\sigma(B_{1/2})} \leq C,
\]
from which
\begin{equation}
\label{eq.123}
[u]_{C^\sigma(B_{\rho/2})} = \rho^{-\sigma} [u(\rho\cdot)]_{C^{\sigma}(B_{1/2})} = \rho^{-\sigma+1-\varepsilon} [v_\rho]_{C^\sigma(B_{1/2})} \leq C\rho^{-\sigma+1-\varepsilon}.
\end{equation}

Now, given $e\in \Sp^{n-1}$ with $e_n = 0$, and for any $h > 0$, define
\[
w (x) = \frac{u(x+e h)-u(x)}{h^\sigma}.
\]
By \eqref{eq.123},
\[
\|w\|_{L^\infty(B_R)} \leq CR^{-\sigma+1-\varepsilon}\quad\textrm{for all}\quad R \geq 1.
\]
We also have
\begin{equation}
\left\{\begin{array}{rcll}
(-L+b\cdot\nabla )w & =& 0& \quad \textrm{in}\quad \R^n_+\\
w & = & 0 & \quad \textrm{in}\quad \R^n_-, \\
\end{array}\right.
\end{equation}
thanks to the fact that $e$ does not have component in the $n$-th direction, $e_n = 0$.

Repeat the previous argument applied to $w$ instead of $u$, to get
\[
[w]_{C^\sigma(B_R)} \leq CR^{-2\sigma+1-\varepsilon}\quad \textrm{for all} \quad R\geq 1.
\]

Repeating iteratively we get that, for $m = \lfloor \frac{1-\varepsilon}{\sigma}+1  \rfloor$, then
\[
[w_m]_{C^\sigma(B_R)} \leq CR^{-m\sigma+1-\varepsilon}\quad \textrm{for all} \quad R\geq 1,
\]
where $w_m$ is an incremental quotient of order $m$ of $u$. Letting $R\to \infty$ we observe that $w_m \equiv 0$.

Since $w_m$ is any incremental quotient of order $m$, this means that for any fixed $x$, $q_x(y') := u(x+(y',0))$ for $y'\in \R^{n-1}$ is a polynomial of order $m-1$ in the $y'$ variables. However, from the growth condition on $u$ the polynomial must grow less than linearly at infinity, and therefore it is constant. This means that for any $x$, $u(x+eh) = u(x)$ for all $h \in \R$ and for all $e\in \Sp^{n-1}$ with $e_n= 0$; i.e., $u(x) = u(x_n)$, as we wanted to see.

Now we can proceed as in the proof of the classification theorem, Theorem~\ref{thm.clas}, and use the classification of 1-dimensional solutions from Proposition~\ref{prop.1DL_2}.
\end{proof}

\subsection{Proof of Theorem~\ref{thm.expansion_}}

We now prove the following result, which will directly yield Theorem~\ref{thm.expansion_}. For this, we combine the ideas in \cite{RS16} with Propositions~\ref{prop.boundestimatesg} and \ref{prop.1DL_2}.

\begin{prop}
\label{prop.expansion}
Let $L$ be an operator of the form \eqref{eq.L}-\eqref{eq.L.cond}, and let $b\in \R^n$. Let $\Gamma$ be a $C^{1,\alpha}$ graph splitting $B_1$ into $U^+$ and $U^-$ (see Definition~\ref{defi.split}), and suppose $0\in \Gamma$ and that $\nu(0) = e_n$, where $\nu(0)$ is the normal vector to $\Gamma$ at 0 pointing towards $U^+$.

Let $f\in L^\infty (U^+)$, and suppose $u\in L^\infty(\R^n)$ satisfies
\begin{equation}
\left\{\begin{array}{rcll}
(-L+b\cdot\nabla )u& = & f & \quad \textrm{in}\quad U^+\\
u& = &0 & \quad \textrm{in}\quad U^-. \\
\end{array}\right.
\end{equation}

Let us denote $\gamma :=  \gamma\left(\frac{b\cdot \nu(0)}{\chi(\nu(0))}\right) = \gamma(b_n/\chi(e_n))$ and $\chi = \chi(e_n)$ as defined in \eqref{eq.gamma}-\eqref{eq.chi}, and suppose that $\gamma\in \left[\gamma_0,\gamma_0\left(1+\frac{\alpha}{8}\right)\right]$ for some $\gamma_0 \in(0,1)$ such that $\gamma_0\left(1+\frac{\alpha}{4}\right)< 1$. Suppose also that $\eta_\nu$ as defined in \eqref{eq.etanu} satisfies $\eta_\nu \leq \frac{\alpha\gamma_0}{64}$, and let $\Upsilon = \gamma_0\left(1+\frac{\alpha}{4}\right)$.

Then, there exists $Q$ with $|Q|\leq C\left(\|u\|_{L^\infty(\R^n)} + \|f\|_{L^\infty(U^+)}\right)$ such that
\[
\big| u(x) - Q (x_n)_+^{\gamma} \big| \leq C\left(\|u\|_{L^\infty(\R^n)} + \|f\|_{L^\infty(U^+)}\right)|x|^\Upsilon
\quad \textrm{for all}\quad x \in B_1,
\]
where the constant $C$ depends only on $n$, $\alpha$, the $C^{1,\alpha}$ norm of $\Gamma$, $\gamma_0$, the ellipticity constants, and $\|b\|$.
\end{prop}

Before proving the previous result let us state a useful lemma. It can be found in \cite[Lemma 5.3]{RS16}.

\begin{lem}[\cite{RS16}]
\label{lem.RS16}
Let $1>\Upsilon > \beta_0\geq \beta$ and $\nu \in \Sp^{n-1}$ some unit vector. Let $u\in C(B_1)$ and define
\[
\phi_r (x) := Q_*(r) (x\cdot \nu)^\beta_+,
\]
where
\[
Q_*(r) := {\rm arg~min}_{Q\in \R} \int_{B_r} \big(u(x) - Q(x\cdot \nu)^\beta_+\big)^2dx = \frac{\int_{B_r} u(x) (x\cdot \nu)^\beta_+ dx}{\int_{B_r} (x\cdot \nu)^{2\beta}_+ dx}.
\]

Assume that for all $r \in (0,1)$ we have
\[
\|u - \phi_r\|_{L^\infty(B_r)} \leq C_0r^\Upsilon.
\]
Then, there is $Q\in \R$ with $|Q| \leq C(C_0 + \|u\|_{L^\infty(B_1)})$ such that
\[
\|u- Q(x\cdot \nu)^\beta_+\|_{L^\infty(B_r)}\leq C C_0 r^\Upsilon
\]
for some constant $C$ depending only on $\Upsilon$ and $\beta_0$.
\end{lem}

We can now prove Proposition~\ref{prop.expansion}.

\begin{proof}[Proof of Proposition~\ref{prop.expansion}]
Let us argue by contradiction. Suppose that there are sequences $\Gamma_i$, $U^+_i$, $U^-_i$, $L_i$, $b_i$, $u_i$, and $f_i$ that satisfy the assumptions
\begin{enumerate}[$~~~~~~~~\bullet~$]
\item $\Gamma_i$ is a $C^{1,\alpha}$ graph with bounded $C^{1,\alpha}$ norm independently of $i$, splitting $B_1$ into $U^+_i$ and $U^-_i$ with $0\in \Gamma_i$ and with $e_n$ being the normal vector at 0 pointing towards $U^+_i$.
\item $L_i$ are of the form \eqref{eq.L}-\eqref{eq.L.cond}, and $\|b_i\|=\|b\|$;
\item For each $\Gamma_i$, the corresponding $\eta_\nu$ as defined in \eqref{eq.etanu} fulfils $\eta_\nu \leq (\alpha\gamma_0)/64$;
\item $\|u_i\|_{L^\infty(\R^n)} + \|f_i\|_{L^\infty(U^+)} = 1$;
\item $u_i$ solves $(-L_i + b_i\cdot \nabla) u_i = f_i$ in $U^+_i$, $u_i = 0$ in $U^-_i$;
\item If we define $\gamma_i := \gamma(b_i\cdot e_n / \chi_i)$ with $\gamma$ as in \eqref{eq.gamma} and $\chi_i = \chi_i(e_n)$ as in \eqref{eq.chi} with the operator $L_i$, then $\gamma_i \in [\gamma_0,\gamma_0(1+\alpha/8)]$;
\end{enumerate}
but they are such that for all $C > 0$ there exists some $i$ such that there is no constant $Q$ satisfying
\[
\big| u_i(x) - Q (x_n )_+^{\gamma_i} \big| \leq C|x|^\Upsilon
\quad \textrm{for all}\quad x \in B_1.
\]
\\
{\it Step 1: Construction and properties of the blow up sequence.}

Let us denote
\[
\phi_{i,r} := Q_{i}(r) (x_n)^{\gamma_i}_+,
\]
where
\[
Q_i(r) := {\rm arg~min}_{Q\in \R} \int_{B_r} (u_i(x) - Q(x_n)^{\gamma_i}_+)^2dx = \frac{\int_{B_r} u_i(x) (x_n)^{\gamma_i}_+ dx}{\int_{B_r} (x_n)^{2\gamma_i}_+ dx}.
\]
From Lemma~\ref{lem.RS16} with $\beta = \gamma_i$ and $\beta_0 = \gamma_0(1+ \alpha/8)$ we have that
\[
\sup_i \sup_{r>0} \left\{ r^{-\Upsilon} \|u_i - \phi_{i,r}\|_{L^\infty(B_r)} \right\}= \infty.
\]

Define the monotone function
\[
\theta(r) := \sup_i \sup_{r' > r} \left\{ (r')^{-\Upsilon} \|u_i - \phi_{i,r'}\|_{L^\infty(B_{r'})} \right\}.
\]

Note that for $r > 0$, $\theta(r) < \infty$, and $\theta(r) \to \infty$ as $r \downarrow 0$. Now take a sequences $r_m\downarrow 0$ and $i_m$ such that
\[
(r_m)^{-\Upsilon} \|u_{i_m} - \phi_{i_m, r_m}\|_{L^\infty(B_{r_m})} \geq \frac{\theta(r_m)}{2},
\]
and denote $\phi_m = \phi_{i_m, r_m}$.

Consider now
\[
v_m (x) = \frac{u_{i_m} (r_m x) - \phi_m (r_m x)}{r_m^\Upsilon \theta(r_m)}.
\]

By definition of $\phi_m$ we have the orthogonality condition for all $m \geq 1$,
\begin{equation}
\label{eq.orth}
\int_{B_1} v_m(x) (x_n)^{\gamma_i}_+ dx = 0.
\end{equation}

Note that also from the choice of $r_m$ we have a nondegeneracy condition for $v_m$,
\begin{equation}
\label{eq.nondeg}
\|v_m\|_{L^\infty(B_1)} \geq \frac{1}{2}.
\end{equation}

From the definition of $\phi_{i, r}$, $\phi_{i, 2r} - \phi_{i, r} = \big(Q_i (2r) - Q_i (r) \big) (x_n)^{\gamma_i}_+$ so that
\begin{align*}
|Q_i(2r) - Q_i(r)| r^{\gamma_i} & = \|\phi_{i, 2r} - \phi_{i, r}\|_{L^\infty(B_r)} \\
& \leq  \|\phi_{i, 2r} - u\|_{L^\infty(B_{2r})}+ \|\phi_{i, r} - u\|_{L^\infty(B_r)}\leq Cr^\Upsilon \theta(r).\\
\end{align*}

Proceeding inductively, if $R = 2^N$, then
\begin{equation}
\label{eq.longeq}
\begin{split}
\frac{r^{\gamma_i - \Upsilon} |Q_i(Rr)- Q_i(r)|}{\theta(r)}& \leq \sum_{j = 0}^{N-1} 2^{j(\Upsilon - \gamma_i)} \frac{{(2^jr)}^{\gamma_i - \Upsilon} |Q_i(2^{j+1}r)- Q_i(2^jr)|}{\theta(r)} \\
& \leq C \sum_{j = 0}^{N-1} 2^{j(\Upsilon - \gamma_i)} \frac{\theta(2^j r)}{\theta(r)} \leq C 2^{N(\Upsilon - \gamma_i)} = CR^{\Upsilon -\gamma_i}.
\end{split}
\end{equation}

Thus, we obtain a bound on the growth control of $v_m$ given by
\begin{equation}
\label{eq.growthctrl}
\|v_m\|_{L^\infty(B_R)} \leq CR^\Upsilon\quad \textrm{for all} \quad R \geq 1.
\end{equation}

Indeed,
\begin{align*}
\|v_m\|_{L^\infty(B_R)} & = \frac{1}{\theta(r_m) r_m^\Upsilon} \|u_i - Q_{i_m}(r_m) (x_n)^{\gamma_i}_+\|_{L^\infty(Rr_m)}\\
& \leq \frac{1}{\theta(r_m) r_m^\Upsilon} \|u_i - Q_{i_m}(Rr_m) (x_n)^{\gamma_i}_+\|_{L^\infty(Rr_m)} +\\
&~~~~~~~~~~~~~~~~+\frac{1}{\theta(r_m) r_m^\Upsilon} |Q_{i_m}(Rr_m) - Q_{i_m}(r_m) | (Rr_m)^{\gamma_i}\\
& \leq \frac{R^\Upsilon \theta(Rr_m)}{\theta(r_m)} + CR^\Upsilon,
\end{align*}
and the result follows from the monotonicity of $\theta$.

Notice also that the previous computation in \eqref{eq.longeq} also gives a bound for $Q_i(r)$ given by
\begin{equation}
\label{eq.qbound}
|Q_i(r)| \leq C\theta(r),
\end{equation}
which follows by putting $R = r^{-1}$.

{\it Step 2: Convergence of the blow up sequence.}

In this second step we show that $v_m$ converges locally uniformly in $\R^n$ to some function $v$ satisfying
\begin{equation}
\label{eq.v}
\left\{\begin{array}{rcll}
(-\tilde{L}+ \tilde{b}\cdot\nabla )v& = & 0 & \quad \textrm{in}\quad \R^n_+\\
v& = &0 & \quad \textrm{in}\quad \R^n_-, \\
\end{array}\right.
\end{equation}
for some operator $\tilde{L}$ of the form \eqref{eq.L}-\eqref{eq.L.cond}, $\|\tilde{b}\|= \|b\|$.

To do so, define
\[
U^+_{R, m} := B_R \cap \left( r_m^{-1}U^+_{i_m}\right) \cap \{x_n > 0\},
\]
and suppose that it is well defined by assuming $m$ is large enough so that $Rr_m < 1/2$.

Notice that in $U^+_{R,m}$, $v_m$ satisfies an elliptic equation with drift,
\[
(-L_{i_m} + b_{i_m}\cdot\nabla) v_m (x) = \frac{r_m}{r_m^\Upsilon\theta(r_m)} f_{i_m} (r_m x)\quad \textrm{in}\quad U^+_{R,m},
\]
since we know that $(-L_i+b_i\cdot\nabla) \phi_m = 0$ in $\{x_n > 0\}$. In particular, since $\Upsilon < 1$, the right-hand side converges uniformly to 0 as $r_m \downarrow 0$.

We will now show that
\begin{equation}
\label{eq.uimphi}
\|u_{i_m} - \phi_m\|_{L^\infty(B_r \cap (U^-_{i_m} \cup \R^n_-)} \leq C\theta(r_m) r^{(1+\alpha)\kappa}\quad\textrm{for all} \quad r < 1/4,
\end{equation}
and where the constant $C$ is independent of $m$, and $\kappa := \gamma_0\left(1-\frac{\alpha}{16}\right)$. Notice that $\kappa < \gamma_0 - 2\eta_\nu$, so that we can use the supersolution from Proposition \ref{prop.supersolution} to get
\[
|u_{i_m}| \leq C\left({\rm dist}(x, U^-)\right)^{\kappa},
\]
with $C$ depending only on $n$, the $C^{1,\alpha}$ norm of $\Gamma$, $\alpha$, the ellipticity constants, and $\|b\|$. On the other hand, by definition of $\phi_m$,
\[
|\phi_m(x)|\leq C Q_{i_m}(r_m) \left({\rm dist}(x, \R^n_-)\right)^{\gamma_i}\leq C\theta(r_m) \left({\rm dist}(x, \R^n_-)\right)^{\kappa}\quad \textrm{for all}\quad x\in B_1,
\]
where we used \eqref{eq.qbound}. Finally, since the domain is $C^{1,\alpha}$, we have that
\[
{\rm dist}(x, U^-_{i_m}) \leq Cr^{1+\alpha},\quad {\rm dist}(x, \R^n_-) \leq Cr^{1+\alpha}\quad\textrm{in}\quad B_r \cap (U^-_{i_m} \cup \R^n_-),
\]
where the constant $C$ depends only on the $C^{1,\alpha}$ norm of the domain $U^+_{i_m}$, and therefore, it is independent of $m$. Thus, combining the last two expressions we get \eqref{eq.uimphi}.

Now, from Proposition~\ref{prop.boundestimatesg} we have
\[
\|u_{i_m}\|_{C^\sigma(B_{1/8})} \leq C,
\]
uniformly in $m$, for some $\sigma\in (0,\gamma_0)$.

From the regularity of $\phi_m$ this yields, in particular,
\begin{equation}
\label{eq.interp1}
\|u_{i_m} - \phi_m\|_{C^\sigma\left(B_r \cap (U^-\cup \R^n_-)\right)} \leq C\theta(r_m), 	
\end{equation}
where we have used again the bound \eqref{eq.qbound}.

Thus, interpolating \eqref{eq.uimphi} and \eqref{eq.interp1} there exists some $\sigma_0 < \sigma$ (depending on $\sigma$, $\gamma_0$, and $\alpha$) such that
\[
\|u_{i_m} - \phi_m\|_{C^{\sigma_0} (B_r \cap (U^-_{i_m} \cup \R^n_-))} \leq C\theta(r_m)r^\Upsilon.
\]
Notice that we can do so because $\Upsilon < \kappa(1+\alpha)$. Scaling the previous expression we obtain
\begin{equation}
\label{eq.interp2}
\|v_m\|_{C^{\sigma_0} (B_R\setminus U^+_{R,m})} \leq C(R)\quad\textrm{for all }m\textrm{ with } Rr_m< 1/4,
\end{equation}
for some constant $C(R)$ that depends on $R$, but is independent of $m$.

We now want to apply Proposition~\ref{prop.boundestimatesg} to $v_m$, rescaled to balls $B_R$. Recall that
\[
(-L_{i_m} + b_{i_m}\cdot\nabla) v_m (x) = \frac{r_m}{r_m^\Upsilon\theta(r_m)} f_{i_m} (r_m x)\quad \textrm{in}\quad U^+_{R,m},
\]
and $v_m$ is $C^{\sigma_0}$ outside $U^+_{R,m}$ by \eqref{eq.interp2}. Notice also that the boundary $\de U^+_{R,m}$ has $C^{1,\alpha}$ norm smaller than the $C^{1,\alpha}$ norm of $\Gamma$ thanks to the fact that we are rescaling with smaller $r_m$ and $Rr_m < 1/4$. Thus, Proposition~\ref{prop.boundestimatesg} can be applied and we obtain that there exists some $\sigma'>0$ small such that
\[
\|v_m\|_{C^{\sigma'}(B_{R/2})} \leq C(R)\quad\textrm{for }m\textrm{ with}\quad Rr_m < 1/4.
\]
we have again that the constant $C(R)$ depends on $R$, but is independent of $m$; i.e, we have reached a uniform $C^{\sigma'}$ bound on $v_m$ over compact subsets.

Thus, up to taking a subsequence, $v_m$ converge locally uniformly to some $v$.
\\[0.3cm]
{\it Step 3: Contradiction.} Up to taking a subsequence if necessary, $L_{i_m}$ converges weakly to some operator $\tilde{L}$ of the form \eqref{eq.L}-\eqref{eq.L.cond}, and $b_{i_m}$ converges to some $\tilde{b}$ with $\|\tilde{b}\|=\|b\|$. Notice that, in particular, this means that $\gamma_i$ converges to some $\gamma_*\in [\gamma_0,\gamma_0(1+\alpha/8)]$, and $\gamma_*= \gamma(\tilde{b}\cdot e_n/\tilde{\chi})$, where $\tilde{\chi} = \tilde{\chi}(e_n)$ is the associated constant defined as in \eqref{eq.chi} with the operator $\tilde{L}$.

On the other hand, the domains $U^+_{i_m}$ converge uniformly to $\R^n_+$ over compact subsets by construction. Thus, passing all this to the limit, we reach that $v$ satisfies \eqref{eq.v}.

Now, passing the growth control \eqref{eq.growthctrl} to the limit, we reach
\[
\|v\|_{L^\infty(B_R)}\leq CR^\Upsilon\quad\textrm{for all}\quad R \geq 1,
\]
so that we can apply the Liouville theorem in the half space, Theorem~\ref{thm.liouv}, to get
\[
v(x) = C(x_n)_+^{{\gamma}_*}.
\]

Passing to the limit \eqref{eq.orth} and using this last expression, we obtain $v \equiv 0$. However, by passing \eqref{eq.nondeg} to the limit we get
\[
\|v\|_{L^\infty(B_1)} \geq \frac{1}{2},
\]
a contradiction.
\end{proof}

\begin{proof}[Proof of Theorem~\ref{thm.expansion_}]
The result follows from Proposition~\ref{prop.expansion} applied to small enough balls so that the condition on $\eta_\nu$ is fulfilled. Notice that the constant $\sigma$ cannot go to 0, because $\tilde\gamma(x_0)$ cannot be made arbitrarily small for a given $L$ and $b$.
\end{proof}

\section{Proof of Theorems \ref{thm.1} and \ref{thm.2}}
\label{sec.8}
In this section, we will prove Theorems~\ref{thm.1} and \ref{thm.2}. We already know that if $x_0$ is a regular free boundary point, then the free boundary is $C^{1,\alpha}$ in a neighbourhood. Next, using the results of the previous section, we show that the regular set is open, and that at any regular free boundary point we have \eqref{eq.ur} below.

\begin{prop}
\label{prop.regopen}
Let $L$ be an operator of the form \eqref{eq.L}-\eqref{eq.L.cond}, and let $b\in \R^n$. Let $u$ be a solution to \eqref{eq.pb}-\eqref{eq.pb2}-\eqref{eq.pb3}.

Then the set of regular free boundary points is relatively open. Moreover, around each regular point $x_0$
\begin{equation}
\label{eq.ur}
0<c r^{1+\tilde\gamma(x_0)} \leq \sup_{B_r(x_0)} u \leq Cr^{1+\tilde\gamma(x_0)}\quad\textrm{for all}\quad r\leq 1,
\end{equation}
for some positive constants $c$ and $C$ depending only on $n$, $\|b\|$, and the ellipticity constants. Here, $\tilde\gamma(x_0)$ is given by \eqref{eq.tildegamma} with $\nu(x_0)$ being the normal vector to the free boundary at $x_0$ pointing towards $\{u > 0\}$.
\end{prop}

\begin{proof}
Suppose without loss of generality that $x_0 = 0$ and $\nu(x_0) = e_n$. The free boundary, $\Gamma$, is $C^{1,\alpha}$ in $B_{r_0}$ for some $\alpha, r_0 > 0$ by Proposition~\ref{prop.C1sigma}. Apply now Theorem~\ref{thm.expansion_} to the partial derivative $\de_n u $ around points $z\in B_{r_0/2}\cap \Gamma$. We obtain
\begin{equation}
\label{eq.sigmaappears}
\left|\de_n u (x) - Q(z) \big((x-z)\cdot\nu(z)\big)^{\tilde\gamma(z)}_+\right| \leq C|x-z|^{\tilde\gamma(z) + \sigma},
\end{equation}
for some $\sigma > 0$, and some constant $C$ independent of $z$.
\\
{\it Step 1: Q is continuous and positive at the origin.} Let us first check that $Q$ is a continuous function on the free boundary at $0$. Indeed, suppose it is not continuous, so that there exists a sequence $z_k \to 0$ on the free boundary such that $\lim_{k\to\infty} Q(z_k) = \bar Q \neq Q(0)$. Then, we have
\[
\left|\de_n u (x) - Q(z_k) \big((x-z_k)\cdot\nu(z_k)\big)^{\tilde\gamma(z_k)}_+\right| \leq C|x-z_k|^{\tilde\gamma(z_k) +\sigma}.
\]

Thus, taking limits as $k \to \infty$, for any fixed $x$, we obtain
\[
\left|\de_n u (x) - \bar Q (x_n)^{\tilde\gamma(0)}_+\right| \leq C|x|^{\tilde\gamma(0)+\sigma}.
\]
We have used here that $\nu$ and $\tilde\gamma$ are continuous. On the other hand, we had
\[
\left|\de_n u (x) - Q(0) (x_n)^{\tilde\gamma(0)}_+\right| \leq C|x|^{\tilde\gamma(0)+\sigma},
\]
so that
\[
|\bar Q -  Q(0)| (x_n)^{\tilde\gamma(0)}_+ \leq C|x|^{\tilde\gamma(0)+\sigma}.
\]
Now take $x = (0,t)\in \R^{n-1}\times\R$ for $t\in\R^+$ and let $t\to 0$. It follows $\bar Q = Q(0)$, a contradiction; i.e., $Q$ is continuous at 0.

We now prove that $Q(0) > 0$ (notice that we already know that $Q(0) \geq 0$ because $u \geq 0$). To do so, we proceed by creating an appropriate subsolution using Lemma~\ref{lem.subsolution}.

First of all, consider a fixed bounded strictly convex $C^{1,\alpha}$ domain $P\subset\{u > 0\}$ touching the free boundary at 0, similar to the domains considered in the proof of Proposition~\ref{prop.boundestimatesg}. Suppose that $P$ has diameter less than 1, and take an $h > 0$ such that, if we denote $\nu_P(z)$ the normal vector to $\de P$ pointing towards the interior of $P$ at $z\in \de P$, then
\[
\tilde{\gamma}_h := \max\left\{\gamma\left(\frac{b\cdot\nu_P(z)}{\chi(\nu_P(z))}\right) \quad \textrm{for}\quad z\in \de P\cap \{x_n < h\} \right\}\leq \tilde\gamma(0) + \frac{\sigma}{4},
\]
where $\sigma$ is the small constant following from Theorem~\ref{thm.expansion_} that appears in \eqref{eq.sigmaappears}.
Let us call
\[
\eta_\nu^{(h)} := \tilde{\gamma}_h-\tilde\gamma(0) \geq 0
\]
Such $h >0$ exists because $P$ is $C^{1,\alpha}$, and $\gamma$ and $\chi$ are continuous. Take now $\kappa =  \tilde\gamma(0) + 3\eta_\nu^{(h)}$, and let $\varrho$ be a regularised distance to $\R^n\setminus P$ as in Definition~\ref{defi.rendist}. In particular, $\varrho \equiv 0$ in $\R^n \setminus P$. We will see that $\phi := \varrho^\kappa \leq C\de_n u$ for an appropriate~$C$.

By Lemma~\ref{lem.subsolution} used in $B_h$ we get that for some constant $\delta_0 < h/2$,
\[
(-L+b\cdot\nabla )\phi \leq -1  \quad \textrm{in}\quad B_{h/2}\cap \{x : 0<d(x, \R^n\setminus P) \leq \delta_0\}.
\]

Now, since $P$ is strictly convex, we have that there exists some $\delta_P$ with $0< \delta_P\leq \delta_0$ such that
\[
(-L+b\cdot\nabla )\phi \leq -1  \quad \textrm{in}\quad \{0<x_n <\delta_P\}\cap P.
\]

Now consider $v_r$ as the one defined in Proposition~\ref{prop.regpt} (there it is called $v$),
\[
v_r(x) = \frac{u(rx)}{r\|\nabla u \|_{L^\infty(B_r)}}.
\]

By the same reasoning as in the proof of Proposition~\ref{prop.fblip} rescaling to a larger ball we have that
\[
\tilde{w}_r = C_1 (\de_n v_r)\chi_{B_2} \geq 0
\]
for $r$ small enough.

From Proposition~\ref{prop.regpt} we can choose $r$ small enough so that for some positive constant $c$,
\[
\tilde{w}_{r} > c > 0\quad \textrm{in}\quad P\cap \{x_n \geq \delta_P\}.
\]

Moreover, also proceeding as in the proof of Proposition~\ref{prop.fblip}, $(-L+b\cdot \nabla)\tilde{w}_r > -\eta$ in $B_{1}\cap \{v_r > 0\}$ for some arbitrarily small constant $\eta$, making $r$ even smaller if necessary. Thus, we can assume
\[
(-L+b\cdot \nabla)\tilde{w}_r > -\frac{\tilde{c}}{2}\quad\textrm{in}\quad B_{1}\cap \{v_r > 0\},
\]
for some $0<\tilde{c}< c$ to be chosen later.

Now compare the functions $\phi$ and $\tilde{c}^{-1}\tilde{w}_r$. Notice that in $\R^n\setminus P$, $\tilde{w}_r \geq \phi \equiv 0$. In $P\cap \{x_n \geq \delta_P\}$, $\tilde{c}$ can be chosen small enough depending on $\delta_P$ and $P$ so that $\tilde{c}^{-1} \tilde{w}_r \geq \phi$ there, because $\tilde{w}_{r} > c > 0$ in $P\cap \{x_n \geq \delta_P\}$. Finally,
\[
(-L+b\cdot\nabla )\phi \leq (-L+b\cdot\nabla )\tilde{w}_r   \quad \textrm{in}\quad \{0<x_n <\delta_P\}\cap P.
\]

Thus, by the maximum principle, for this particular $r$ fixed we have that $\tilde{w}_r \geq \tilde{c}\phi$. Going back to the definition of $\tilde{w}_r$, this means that for some $\rho$ and $c$ positive constants
\[
\de_n u(te_n) \geq c \varrho(te_n)\quad \textrm{for}\quad 0<t<\rho.
\]
For $\rho$ small enough, $\varrho$ is comparable to $(x_n)_+^\kappa$ along the segment $te_n$, so that we actually have
\begin{equation}
\label{eq.orddn}
\de_n u(te_n) \geq c t^\kappa\quad \textrm{for}\quad 0<t<\rho.
\end{equation}
Now, if $Q(0) = 0$ then
\[
|\de_n u (x) |\leq C|x|^{\tilde\gamma(0) + \sigma}.
\]
Since $\kappa < \tilde\gamma(0)+\sigma$ we get a contradiction with \eqref{eq.orddn}. Thus, $Q(0) > 0$.
\\[0.3cm]
{\it Step 2: Conclusion of the proof.} For $z\in \Gamma\cap B_r$ for $r$ small enough we have that $Q(z) > 0$, because $Q$ is continuous and $Q(0) > 0$. In particular,
\[
\left|\de_n u (x) - Q(z) \big((x-z)\cdot\nu(z)\big)^{\tilde\gamma(z)}_+\right| \leq C|x-z|^{\tilde\gamma(z)+\sigma}.
\]

By taking $x = z+te_n$ for $t > 0$ we get
\[
\left|\de_n u (z+te_n) - Q(z) \big(\nu_n(z) t\big)^{\tilde\gamma(z)}_+\right| \leq Ct^{\tilde\gamma(z)+\sigma}.
\]

Integrating with respect to $t$ from $0$ to $t'<1$, using that $\de_nu (z) = 0$ and $\nu_n(z) > 1/2$ for $r$ small enough and recalling that $Q(z)> 0$, we get
\[
u (z+t'e_n) \geq ct'^{1+\tilde\gamma(z)} > 0 ,
\]
so that in particular, $z$ is a regular point; i.e., the set of regular points is relatively open. Doing the same for $z = 0$ we get one of the inequalities from \eqref{eq.ur},
\begin{equation}
\label{eq.ur2}
\sup_{B_r} u \geq c r^{1+\tilde\gamma(0)} >0 \quad\textrm{for all}\quad r\leq 1.
\end{equation}

On the other hand, we can also find the expansion at 0 for $\de_i u$ for any $i\in \{1,\dots,n\}$,
\[
\left|\de_i u (x) - Q_i (x_n)^{\tilde\gamma(0)}_+\right| \leq C|x|^{\tilde\gamma(0)+\sigma}.
\]
Therefore,
\[
|\nabla u (x)|\leq C\left( |x|^{\tilde\gamma(0)} + |x|^{\tilde\gamma(0)+\sigma}\right).
\]
Integrating, and using $\nabla u(0) = 0$
\[
u (x)\leq C\left( |x|^{1+\tilde\gamma(0)} + |x|^{1+\tilde\gamma(0)+\sigma}\right),
\]
i.e.,
\[
\sup_{B_r} u \leq Cr^{1+\tilde\gamma(0)}\quad\textrm{for all}\quad r\leq 1.
\]
Thus, combined with \eqref{eq.ur2}, this proves \eqref{eq.ur}.
\end{proof}

\begin{prop}
\label{prop.regopen2}
Let $L$ be an operator of the form \eqref{eq.L}-\eqref{eq.L.cond}, and let $b\in \R^n$. Let $u$ be a solution to \eqref{eq.pb}-\eqref{eq.pb2}-\eqref{eq.pb3} and let $x_0$ be a free boundary regular point. Then
\begin{equation}
u(x) = c_0\big((x-x_0)\cdot \nu(x_0)\big)^{1+\tilde\gamma(x_0)}_+ + o\left(|x-x_0|^{1+\tilde\gamma(x_0)+ \sigma}\right)
\end{equation}
with $c_0 > 0$ and for some $\sigma > 0$. Here $\tilde\gamma(x_0)$ is given by \eqref{eq.tildegamma}, with $\nu(x_0)$ being the normal vector to the free boundary at $0$ pointing towards $\{u > 0\}$; and $\sigma$ depends only on $n$, the ellipticity constants, and $\|b\|$.
\end{prop}
\begin{proof}
Assume that $x_0 = 0$ and $\nu(x_0) = e_n$. From the expansions in the proof of Proposition~\ref{prop.regopen} we have
\begin{equation}
\label{eq.intt}
\de_i u(x) = Q_i(x_n)^{\tilde\gamma(0)}_+ + o\left(|x|^{\tilde\gamma(0)+\sigma}\right),
\end{equation}
for some $Q_i$, with $Q_n > 0$, and $\sigma > 0$. Now, let $x = (x',x_n)$, with $x' \in \R^{n-1}$ and $x_n \in \R$. Integrating the expression \eqref{eq.intt} in the segment with endpoints $0$ and $(x',0)$ we get
\[
u(x',0) = o\left(|x|^{1+\tilde\gamma(0)+\sigma}\right).
\]
Then, integrating in the segment with endpoints $(x',0)$ and $(x',x_n)$ we find
\[
u(x',x_n) = \frac{Q_n}{1+\tilde\gamma(0)}(x_n)^{1+\tilde\gamma(0)}_+ + o\left(|x|^{1+\tilde\gamma(0)+\sigma}\right).
\]
Thus, \eqref{eq.ur2} is proved.
\end{proof}

We finally can put all elements together to prove our main results, Theorems~\ref{thm.1} and \ref{thm.2}.

\begin{proof}[Proof of Theorem~\ref{thm.2}]
After subtracting the obstacle and dividing by a constant, we can assume $u$ is a solution to \eqref{eq.pb}-\eqref{eq.pb2}-\eqref{eq.pb3}. Then the result we want is a combination of Propositions~\ref{prop.C1sigma}, \ref{prop.regopen}, and \ref{prop.regopen2}.
\end{proof}

\begin{proof}[Proof of Theorem~\ref{thm.1}]
It is a particular case of Theorem~\ref{thm.2}; we only need to check that $\chi \equiv 1$. For this, notice that the kernel is constant and given by $\mu(\theta) = c_{n,1/2}$, where the constant $c_{n,s}$ is the one appearing in the definition of fractional Laplacian,
\[
c_{n,s} := \left(\int_{\R^n} \frac{1-\cos(x_1)}{|x|^{n+2s}}dx\right)^{-1};
\]
see for example \cite{DPV12}. Thus, the value of $\chi$ for $(-\Delta)^{1/2}$ is
\[
\chi(e) = \frac{\pi c_{n,1/2}}{2}\int_{\Sp^{n-1}} |\theta\cdot e| d\theta.
\]

Notice that, by changing variables to polar coordinates,
\begin{align*}
c_{n,1/2}^{-1} = \int_{\R^n} \frac{1-\cos(x_1)}{|x|^{n+1}}dx = \int_{\Sp^{n-1}}\int_0^\infty \frac{1-\cos(r\theta_1)}{r^2} dr d\theta = \frac{\pi}{2}\int_{\Sp^{n-1}}|\theta_1|d\theta,
\end{align*}
where we have used that $\int_0^\infty (1-\cos(t))t^{-2} dt = \pi/2$. This immediately yields that $\chi \equiv 1$ for $(-\Delta)^{1/2}$, as desired.
\end{proof}

We next prove the almost optimal regularity of solutions. Given an operator $L$ of the form \eqref{eq.L}-\eqref{eq.L.cond}, the associated $\chi$ defined as in \eqref{eq.chi}, and $b\in \R^n$, we define
\begin{equation}
\label{eq.mingamma}
\gamma^-_{L,b} := \inf_{e\in \Sp^{n-1}} \gamma\left(\frac{b\cdot e}{\chi(e)}\right),
\end{equation}
where $\gamma$ is given by \eqref{eq.gamma}. Notice that $\gamma_{L,b}^- \in (0,1/2]$.
\begin{prop}
\label{prop.almostoptimal}
Let $L$ be an operator of the form \eqref{eq.L}-\eqref{eq.L.cond}, and let $b\in \R^n$. Let $u$ be a solution to \eqref{eq.pb}-\eqref{eq.pb2}-\eqref{eq.pb3}. Then, for any $\varepsilon > 0$,
\[
\|u\|_{C^{1,\gamma^-_{L,b}-\varepsilon}(\R^n)} \leq C_\varepsilon,
\]
where the constant $C_\varepsilon$ depends only on $n$, $L$, $b$, and $\varepsilon$. The constant $\gamma^-_{L,b}$ is given by \eqref{eq.mingamma}.
\end{prop}
\begin{proof}
In order to prove the bound we first check the growth of the solution at the free boundary, and then we combine it with interior estimates.

For simplicity, we will denote $\gamma_\varepsilon = \gamma_{L,b}^- - \varepsilon$.

{\it Step 1: Growth at the free boundary.} We first prove that, if 0 is a free boundary point, then
\begin{equation}
\label{eq.firststep}
\sup_{r > 0} \frac{\|\nabla u\|_{L^\infty(B_r)}}{r^{\gamma_\varepsilon}} \leq C,
\end{equation}
for some constant $C$ depending only on $n$, $L$, $b$, and $\varepsilon$.

We proceed by contradiction, using a compactness argument. Suppose that it is not true, so that there exists a sequence of functions $u_k$, $f_k$, with $\|u_k\|_{C^{1,\tau}}\leq 1$ for some $\tau > 0$ fixed and $\|f_k\|_{C^1(\R^n)}\leq 1$, such that
\begin{equation}
\left\{\begin{array}{rcll}
u_k & \geq &  0 & \quad \textrm{in}\quad \R^n\\
(-L+b\cdot\nabla) u_k & \leq & f_k & \quad \textrm{in}\quad \R^n \\
(-L+b\cdot\nabla) u_k & = & f_k & \quad \textrm{in}\quad \{u_k>0\} \\
D^2u_k & \geq & -1 & \quad \textrm{in} \quad \R^n,\\
\end{array}\right.
\end{equation}
but $u_k$ are such that
\[
\theta(r) := \sup_{i}\, \sup_{r' > r}\,\,(r')^{-\gamma_\varepsilon} \|\nabla u_k\|_{L^\infty(B_{r'})} \to \infty \quad\textrm{as}\quad r\downarrow 0.
\]

Notice that for $r > 0$, $\theta(r) < \infty$ and that $\theta$ is a monotone function, with $\theta(r) \to \infty$ as $r \downarrow 0$.
Now take sequences $r_m \downarrow 0$ and $i_m$ such that
\[
r_m^{-\gamma_\varepsilon}\|\nabla u_{i_m}\| \geq \frac{\theta(r_m)}{2},
\]
and define the functions
\[
v_m(x) := \frac{u_{i_m}(r_m x)}{r_m^{1+\gamma_\varepsilon}\theta(r_m)}.
\]

Notice that
\begin{equation}
\label{eq.nondegfin}
\|\nabla v_m\|_{L^\infty(B_1)} \geq \frac{1}{2},
\end{equation}
and
\begin{equation}
\label{eq.clasfin}
D^2 v_m \geq -\frac{r_m^{1-\gamma_\varepsilon}}{\theta(r_m)}\quad\textrm{in}\quad \R^n,\quad\quad |(L+b\nabla)(\nabla v_m)| \leq \frac{r_m^{1-\gamma_\varepsilon}}{\theta(r_m)}\quad\textrm{in}\quad \{v_m > 0\}.
\end{equation}

On the other hand,
\begin{equation}
\label{eq.grofin}
\|\nabla v_m\|_{L^\infty(B_R)} = \frac{\|\nabla u_{i_m}\|_{L^\infty(B_{Rr_m})}}{r_m^{\gamma_\varepsilon}\theta(r_m)}  \leq R^{\gamma_\varepsilon} \frac{\theta(Rr_m)}{\theta(r_m)} \leq R^{\gamma_\varepsilon}\quad\textrm{for}\quad R \geq 1.
\end{equation}

Therefore, noticing that $r_m^{1-\gamma_\varepsilon}/\theta(r_m) \to 0$ as $m\to \infty$, we can apply Proposition~\ref{prop.reg.u} to deduce that, for some $\tau > 0$ independent of $m$,
\[
\|v_m\|_{C^{1,\tau}(B_R)}\leq C(R),
\]
for some constant depending on $R$, $C(R)$. Let us take limits as $m\to \infty$. By Arzelà-Ascoli, $v_m$ converges, up to taking a subsequence, in $C^{1}_{\rm loc}(\R^n)$ to some $v_\infty$. By taking to the limit the properties \eqref{eq.clasfin}-\eqref{eq.grofin} we reach that $v_\infty$ should be a convex global solution. By the classification theorem, Theorem~\ref{thm.clas}, we have that either $v\equiv 0$
\[
v_\infty(x) = C(e\cdot x)_+^{1+\gamma(b\cdot e/\chi(e))}\quad\textrm{for some}\quad e\in \Sp^{n-1},
\]
where $\gamma$ and $\chi$ are given by \eqref{eq.gamma}-\eqref{eq.chi}. Notice, however, that taking \eqref{eq.grofin} to the limit, $v_\infty$ grows at most like $\gamma_\varepsilon$, and by definition $\gamma(b\cdot e/\chi(e)) >  \gamma_\varepsilon$. Therefore, we must have $v_\infty \equiv 0$. But this is a contradiction with \eqref{eq.nondegfin}  in the limit. Therefore, we have proved \eqref{eq.firststep}.

{\it Step 2: Conclusion.} Let us combine the previous growth with interior estimates to obtain the desired result.

Let $x, y\in \R^n$, let $r = |x-y|$ and $R = {\rm dist}(x, \{u = 0\})$. We want to prove that for some constant $C_\varepsilon$ then
\[
|\nabla u(x)-\nabla u(y)|\leq Cr^{\gamma_\varepsilon}.
\]

Without loss of generality and by the growth found in the first step we can assume that $x, y\in \{u > 0\}$. Let $\bar x\in \de\{u = 0\}$ be such that ${\rm dist}(\bar x, x) = R$. We separate two cases:
\begin{enumerate}[$\bullet~~$]
\item If $4r > R$,
\begin{align*}
|\nabla u (x)-\nabla u(y)|& \leq |\nabla u (x)-\nabla u(\bar x)|+|\nabla u (\bar x)-\nabla u(y)|  \\
& \leq C\big(R^{\gamma_\varepsilon} + (R+r)^{\gamma_\varepsilon}\big) \leq Cr^{\gamma_\varepsilon},
\end{align*}
where we have used the growth found in Step 1.
\item If $4r \leq R$, then $x, y \in B_{R/2}(x)$, and $B_R(x)\subset \{u > 0\}$. Notice that we have
\[
(-L+b\cdot\nabla)(\nabla u) = \nabla f\quad\textrm{in}\quad B_R(x).
\]
From the interior estimates in Proposition~\ref{prop.intest} rescaled,  we have
\[
R^{\gamma_\varepsilon} [\nabla u]_{C^{\gamma_\varepsilon} (B_{R/2}(x))} \leq C\left(R \|\nabla f\|_{L^\infty(B_R(x))} + \|\nabla u\|_{L^\infty(B_R(x))}+ \int_{\R^n} \frac{|\nabla u (Rx)|}{1+|x|^{n+1}}\right).
\]

Now notice that thanks to the growth found in Step 1 we have, on the one hand,
\[
\|\nabla u\|_{L^\infty(B_R(x))} \leq CR^{\gamma_\varepsilon},
\]
and on the other hand,
\[
\int_{\R^n} \frac{|\nabla u (Rx)|}{1+|x|^{n+1}} \leq R^{\gamma_\varepsilon} \int_{\R^n} \frac{|x|^{\gamma_\varepsilon}}{1+|x|^{n+1}} = CR^{\gamma_\varepsilon},
\]
so that putting all together and using $\|\nabla f\|_{L^\infty(\R^n)} \leq 1$, it yields,
\[
[\nabla u]_{C^{\gamma_\varepsilon} (B_{R/2}(x))} \leq C\left(1+R^{1-\gamma_\varepsilon}\right).
\]
Thus, if $R \leq 4$ we are done. Now suppose $R > 4$. If $r < 1$, by applying interior estimates to $B_1(x)$ we are done. If $r \geq 1$, we are also done, because $|\nabla u(x)-\nabla u(y)|\leq 2\|\nabla u\|_{L^\infty(\R^n)} \leq C$.
\end{enumerate}
Thus, we have reached the desired result.
\end{proof}

As a consequence, we have the following immediate corollary.

\begin{cor}
\label{cor.2}
Let $L$ be an operator of the form \eqref{eq.L}-\eqref{eq.L.cond}, and let $b\in \R^n$. Let $u$ be the solution to \eqref{eq.obstpb} for a given obstacle $\varphi$ of the form \eqref{eq.obst}. Then, for any $\varepsilon > 0$,
\[
\|u\|_{C^{1,\gamma_{L,b}^- - \varepsilon}(\R^n)}\leq C_\varepsilon,
\]
where $C_\varepsilon$ depends only on $n$, $L$, $b$, $\varepsilon$, and $\|\varphi\|_{C^{2,1}(\R^n)}$. The constant $\gamma_{L,b}^-$ is given by \eqref{eq.mingamma}.
\end{cor}
\begin{proof}
After subtracting the obstacle and dividing by an appropriate constant, we can apply Proposition~\ref{prop.almostoptimal} and the result follows.
\end{proof}

Finally, we prove Corollary~\ref{cor.1}.

\begin{proof}[Proof of Corollary~\ref{cor.1}]
After subtracting the obstacle and dividing by a constant, we get that this result is a particular case of Proposition~\ref{prop.almostoptimal}, but the constant $C_\varepsilon$ depends on $b$ and not only on $\|b\|$.

To prove that $C_\varepsilon$ actually depends on $\|b\|$, the proof of Proposition~\ref{prop.almostoptimal} can be rewritten by taking also sequences of vectors $b_k\in \R^n$ with $\|b_k\| = \|b\|$; by compactness, up to a subsequence they converge to some $\tilde b$ with $\|\tilde b\| = \|b\|$ and the rest of the proof is the same.
\end{proof}

\section{A nondegeneracy property}
\label{sec.9}
In the obstacle problem for the fractional Laplacian (without drift), in \cite{BFR15}, Barrios, Figalli and the second author proved a non-degeneracy condition at all free boundary points for obstacles satisfying $\Delta \varphi \leq 0$. From this, and by means of a Monneau-type monotonicity formula, they establish a global regularity result for the free boundary.

In the obstacle problem with critical drift for the fractional Laplacian we can actually find a non-degeneracy result analogous to the one found in \cite{BFR15}. In this case, however, we cannot establish regularity of the singular set, since we do not have (and do not expect) any monotonicity formula for this problem.

\begin{prop}
\label{prop.nondeg}
Let $b \in \R^n$, and suppose that $\varphi\in C^{1,1}(\R^n)$. Assume that $\varphi$ is concave in $\{\varphi > 0\}$ or, more generally, that
\[
(\Delta+\de_{bb}^2)\,\varphi \leq 0 \quad\textrm{in} \quad\{\varphi > 0\},\quad\varnothing \neq \{\varphi > 0\} \Subset \R^n.
\]
Let $u$ be a solution to the obstacle problem \eqref{eq.obstpb_st}. Then, there exist constants $c, r_0 > 0$ such that for any $x_0$ a free boundary point then
\[
\sup_{B_r(x_0)} (u-\varphi) \geq c r^2\quad\textrm{for all}\quad 0 < r < r_0.
\]
\end{prop}
\begin{proof}
Let $w := \big(	(-\Delta)^{1/2} + b\cdot \nabla\big) u$, so that $w \geq 0$. If $w\equiv 0$, by the interior estimates rescaled, and using that $u$ is globally bounded, we reach $u$ is constant. From $\lim_{|x|\to \infty} u(x) = 0$ we would get $u \equiv 0$, but this is a contradiction with $\varnothing \neq \{\varphi > 0\}$. Thus, $w\not\equiv 0$.

Notice, however, that $w\equiv 0$ in $\{u > \varphi\}$. In particular, given $\bar x\in \{u > \varphi\}$, then $\nabla w(\bar x) = 0$ and $w$ has a global minimum at $\bar x$, so that
\[
\big((-\Delta)^{1/2} - b\cdot\nabla\big) w (\bar x) = (-\Delta)^{1/2} w (\bar x) < 0.
\]

Now, noticing that $\{\varphi > 0\}\Subset\R^n$, we get that by compactness there are some $\bar c, \bar r> 0$ such that for any $\bar x\in \{u > \varphi\}$ with ${\rm dist}(\bar x, \{u = \varphi\}) \leq \bar r$ then
\[
\big((-\Delta)^{1/2} - b\cdot\nabla\big) w (\bar x) \leq -\bar c < 0.
\]

Now, since $\big((-\Delta)^{1/2} +b\cdot \nabla\big) u = w$ in $\R^n$ and from the semigroup property of the fractional Laplacian,
\[
-\Delta u - b_ib_j\de_{ij}u = \big((-\Delta)^{1/2} - b\cdot\nabla\big) w \leq -\bar c\quad\textrm{in}\quad\bar U,
\]
where $\bar U := \{u > \varphi\}\cap \{{\rm dist}(\cdot, \{u = \varphi\}) \leq \bar r \}$. Note that the operator $\Delta + b_ib_j\de_{ij}$ is uniformly elliptic, with ellipticity constants 1 and $1+\|b\|^2$.

Since $u > 0$ on the contact set, by compactness there exists some $h > 0$ such that $\varphi \geq h$ in $\{u = \varphi\}$. By continuity, there exists some $0<r_0 < \bar r/2$ such that
\[
\varphi > 0 \quad\textrm{in}\quad U_0 := \{u > \varphi\}\cap \{{\rm dist}(\cdot, \{u = \varphi\}) \leq 2r_0 \}.
\]

Now let $\bar x \in U_0$ with ${\rm dist}(\bar x, \{u = \varphi\}) \leq r_0$, and consider $r\in (0,r_0)$. From the condition on $\varphi$, $(\Delta+\de_{bb}^2) \varphi \leq 0 \textrm{ in } \{\varphi > 0\}$, we get that if $\bar u := u - \varphi$ then
\[
(\Delta+\de_{bb}^2)\,\bar u	\geq \bar c>0\quad\textrm{in}\quad \{\bar u > 0\}\cap B_r(\bar x)\subset U_0.
\]
Therefore, if we define
\[
v := \bar u - \frac{\bar c}{2(n+\|b\|^2)}|x-\bar x|^2\quad\textrm{in}\quad \{\bar u > 0\}\cap B_r(\bar x),
\]
then
\[
(\Delta+\de_{bb}^2) v \geq 0.
\]

By the maximum principle, if $\Omega_r := \{\bar u > 0\}\cap B_r(\bar x)$ then
\[
0< \bar u (x_1) \leq \sup_{\Omega_r}\,v = \sup_{\de\Omega_r}\,v.
\]
Since $v < 0$ in $\de\{\bar u > 0\} \cap B_r(\bar x)$,
\[
0 < \sup_{\{\bar u > 0\} \cap \de B_r(\bar x)} v \leq \sup_{\de B_r(\bar x)}\bar u - c r^2,
\]
where $c = \frac{\bar c}{2(n+\|b\|^2)}$. Therefore, $c$ is independent of $\bar x$, and we can let $\bar x \to x_0$, to obtain the desired result.
\end{proof}

\end{document}